\documentclass[final,a4paper,UKenglish,cleveref, autoref, thm-restate]{mylipics-v2021}
\pdfoutput=1
\nolinenumbers

\hideLIPIcs

\usepackage{appendix}

\usepackage{amsmath}
\usepackage{amsthm}
\usepackage{amssymb}
\usepackage{stmaryrd}
\usepackage{bussproofs}
\usepackage[all]{xy}
\usepackage[usenames]{xcolor}
\usepackage{amsfonts}
\usepackage[utf8]{inputenc}
\usepackage{hyperref}
\usepackage{latexsym}
\usepackage[noadjust]{cite}
\usepackage{multicol}
\usepackage{mathrsfs} 
\usepackage{bbm}
\usepackage{dsfont}

\usepackage{tikz-cd}
\usepackage{framed}

\tikzset{shiftarr/.style={
        rounded corners,%
        to path={--([#1]\tikztostart.center)
                     -- ([#1]\tikztotarget.center) \tikztonodes
                     -- (\tikztotarget)},
}}

\usepackage[author=anonymous,nomargin,marginclue,footnote]{fixme}
\FXRegisterAuthor{ls}{als}{LS}
\FXRegisterAuthor{cf}{acf}{CF}
\FXRegisterAuthor{sm}{asm}{SM}

\usepackage{hyperref}
\hypersetup{hidelinks,final,bookmarks}

\renewcommand{\sectionautorefname}{Section}%
\renewcommand{\subsectionautorefname}{Section}%
\renewcommand{\subsubsectionautorefname}{Section}%

\newcommand{\resetCurThmBraces}{%
\gdef\curThmBraceOpen{(}%
\gdef\curThmBraceClose{)}}
\resetCurThmBraces
\newcommand{\removeThmBraces}{%
\gdef\curThmBraceOpen{}%
\gdef\curThmBraceClose{}}
\resetCurThmBraces

\usepackage{etoolbox}
\patchcmd{\thmhead}{(#3)}{\curThmBraceOpen #3\curThmBraceClose }{}{}

\usepackage[inline]{enumitem}
\setlist[enumerate,1]{label=(\arabic*),font=\normalfont,align=left,leftmargin=0pt,labelindent=0pt,listparindent=\parindent,labelwidth=0pt,itemindent=!,topsep=3pt,parsep=0pt,itemsep=3pt,start=1}
\setlist[enumerate,2]{label=(\alph*),font=\normalfont,labelindent=*,leftmargin=*,topsep=3pt,start=1}
\setlist[itemize]{labelindent=*,leftmargin=*,topsep=5pt,itemsep=3pt}
\setlist[description]{labelindent=*,leftmargin=*,itemindent=-1 em}

\usepackage[notref,notcite]{showkeys}

\usepackage{seqsplit}
\usepackage{xstring}
\newcommand{\defaultshowkeysformat}[1]{%
\StrSubstitute{#1}{ }{\textvisiblespace}[\TEMP]%
\parbox[t]{\marginparwidth}{\raggedright\normalfont\small\ttfamily\(\{\){\color{red!50!black}\expandafter\seqsplit\expandafter{\TEMP}}\(\}\)}%
}

\renewcommand*\showkeyslabelformat[1]{%
\noexpandarg%
\defaultshowkeysformat{#1}%
}

\numberwithin{equation}{section}

\newcommand{\takeout}[1]{\empty}

\newcommand{\Eq}{\mathop{\mathit{Eq}}}

\newcommand{\into}{\hookrightarrow}
\newcommand{\isdef}{\mathord{\downarrow}}
\newcommand{\RelAx}{\mathcal{A}}
\newcommand{\Vars}{\mathsf{Var}}
\newcommand{\rimpl}{\implies} 
\DeclareMathOperator{\Str}{\mathbf{Str}} 
\newcommand{\Rels}{\Pi}
\newcommand{\power}{\pitchfork} \newcommand{\CJ}{\mathcal{J}}
\newcommand{\BC}{\mathbf{C}} \newcommand{\BV}{\mathbf{V}}

\newcommand{\infrule}[2]{\frac{#1}{#2}}

\newcommand{\arity}{\mathsf{ar}} 
 \newcommand{\EM}{\mathsf{EM}}

 \newcommand{\id}{\mathsf{id}}

\renewcommand{\ell}{\mathsf{l}} \newcommand{\Set}{\mathbf{Set}}

\newcommand{\EAr}{\mathsf{E\text{-}Ar}}
\newcommand{\IAr}{\mathsf{I\text{-}Ar}}

\newcommand{\V}{\mathcal{V}}
\newcommand{\bbT}{\mathbb{T}}
\newcommand{\T}{\mathcal{T}}

\DeclareMathOperator{\card}{\mathsf{card}}
\newcommand{\Pos}{\textbf{Pos}}
\newcommand{\Met}{\textbf{Met}}
\newcommand{\E}{\mathcal{E}}

\newcommand{\colim}{\mathop{\mathsf{colim}}}
\newcommand{\Pres}{\mathsf{Pres}}
\newcommand{\TSigma}{T_\Sigma}
\DeclareMathOperator{\Alg}{\mathsf{Alg}}  
\newcommand{\AlgSigma}{\Alg \Sigma}

\newcommand{\CatC}{\mathscr{C}}

\newcommand{\Horn}{\mathscr{H}}
\newcommand{\pres}{\mathscr{P}}
\newcommand{\Edge}{\mathsf{E}}
\newcommand{\edge}{e}

\newcommand{\Refl}{R}
\newcommand{\ev}{\mathsf{ev}}

\newcommand{\Mnd}{\mathsf{Mnd}}
\newcommand{\sub}{\mathsf{sub}}
\newcommand{\Term}{\mathscr{T}}
\newcommand{\F}{\mathscr{F}}
\newcommand{\Lim}{\mathsf{lim}}
\newcommand{\CMS}{\mathbf{CMS}}

\theoremstyle{plain}
\newtheorem{thm}{Theorem}[section]
\newtheorem{lem}[thm]{Lemma}

\newtheorem{propn}[thm]{Proposition}
\newtheorem{cor}[thm]{Corollary}

\theoremstyle{definition}
\newtheorem{defn}[thm]{Definition}
\newtheorem{expl}[thm]{Example}

\newtheorem{rem}[thm]{Remark}
\newtheorem{notn}[thm]{Notation}

\newtheorem{assn}[thm]{Assumption}

\newcommand{\Id}{\mathsf{Id}}
\newcommand{\subto}{\hookrightarrow}
\newcommand{\inl}{\mathsf{inl}}
\newcommand{\inr}{\mathsf{inr}}

\newcommand{\xra}[1]{\xrightarrow{~#1~}}

\newcommand{\Q}{\mathds{Q}}

\title{Monads on Categories of Relational Structures}

\author{Chase Ford\footnote{Chase Ford acknowledges support by the Deutsche Forschungsgemeinschaft (DFG) as part of the Research and Training Group 2475 ``Cybercrime and Forensic Computing'' (393541319/GRK2475/1-2019)},
Stefan Milius\footnote{Stefan Milius acknowledges by the Deutsche Forschungsgemeinschaft (DFG) under projects MI 717/5-2 and MI 717/7-1}, 
and Lutz Schr\"{o}der}{Friedrich-Alexander-Universit\"{a}t Erlangen-N\"{u}rnberg, Germany}{}{}{
}

\authorrunning{C. Ford and S. Milius and L. Schr\"{o}der} 

\Copyright{Matthew C. Ford and Stefan Milius and Lutz Schr\"{o}der} 

\begin{document}
\maketitle

\begin{abstract}
  We introduce a framework for universal algebra in categories of
  relational structures given by finitary relational signatures and
  finitary or infinitary Horn theories, with the arity~$\lambda$ of a
  Horn theory understood as a strict upper bound on the number of
  premisses in its axioms; key examples include partial orders
  ($\lambda=\omega$) or metric spaces ($\lambda=\omega_1$). We
  establish a bijective correspondence between $\lambda$-accessible
  enriched monads on the given category of relational structures and a
  notion of $\lambda$-ary \emph{algebraic theories} (i.e.~with
  operations of arity~\mbox{$<\lambda$}), with the syntax of algebraic
  theories induced by the relational signature (e.g.~inequations or
  equations-up-to-$\epsilon$).  We provide a generic sound and
  complete derivation system for such \emph{relational algebraic
    theories}, thus in particular recovering (extensions of) recent
  systems of this type for monads on partial orders and metric spaces
  by instantiation. In particular, we present an $\omega_1$-ary
  algebraic theory of metric completion. The theory-to-monad direction
  of our correspondence remains true for the case of $\kappa$-ary
  algebraic theories and $\kappa$-accessible monads for $\kappa<\lambda$,
  e.g.~for finitary theories over metric spaces.
\end{abstract}
\renewcommand{\sectionautorefname}{Section}%
\renewcommand{\subsectionautorefname}{Section}%
\renewcommand{\subsubsectionautorefname}{Section}%
%
%
\section{Introduction}

\noindent Monads play an established role in the semantics of
sequential and concurrent programming~\cite{Moggi91} -- they
encapsulate side-effects, such as statefulness, nontermination,
nondeterminism, or probabilistic branching. The well-known
correspondence between monads on the category of sets and algebraic
theories~\cite{Linton66} impacts accordingly on programming syntax, as
witnessed, for example, in work on algebraic effects~\cite{PlotkinPower01}:
Operations of the theory serve as syntax for computational effects
such as non-deterministic or probabilistic choice. The comparative
analysis of programs or systems beyond two-valued equivalence
checking, e.g.~under behavioural preorders, such as similarity, or
behavioural distances, involves monads based on categories beyond
sets, such as the categories~$\Pos$ of partial orders or~$\Met$ of
($1$-bounded) metric spaces. This has sparked recent interest in
presentations of such monads using suitable variants of the notion of
algebraic theory. While it is, in principle, possible to work with
equational presentations that encapsulate the additional structure
within the signature~\cite{KP93}, it seems at least equally natural to
represent the additional structure (e.g.~distance or ordering) within
the judgements of the theory. Indeed, Mardare, Power, and Plotkin
replace equations with equations-up-to-$\epsilon$ in their
\emph{quantitative algebraic theories}~\cite{MPP16}, which present
monads on~$\Met$, and in our own previous work on behavioural
preorders~\cite{FordEA21} as well as in our own recent work with
Ad\'amek~\cite{AFMS20}, we have used \emph{in}equational theories 
to present monads on~$\Pos$.

In the present paper, we introduce a generalized approach to such
notions of algebraic theory: We work in categories of finitary
relational structures (more precisely, the objects are sets
interpreting a given signature of finitary relation symbols),
axiomatized by Horn theories whose axioms are implications with
possibly infinite sets of antecedents. We say that such a theory is
$\lambda$-ary for a regular cardinal~$\lambda$ if all its axioms have
less than~$\lambda$ antecedents. For instance,~$\Pos$ can be presented
by a finitary (i.e.~$\omega$-ary) Horn theory over a binary
relation~$\le$, and~$\Met$ by an $\omega_1$-ary Horn theory over
binary relations~$=_\epsilon$ `equality up to~$\epsilon$' indexed over
rational numbers $\epsilon$. We exploit that the models of a
$\lambda$-ary Horn theory form a locally $\lambda$-presentable
category~$\CatC$~\cite{AR94} to give a syntactic characterization of
$\lambda$-accessible monads on~$\CatC$ in terms of a notion of
relational algebraic theory, in the sense that we prove a monad-theory
correspondence. Following Kelly and Power~\cite{KP93}, we use
$\lambda$-presentable objects of~$\CatC$ as arities and as contexts of
axioms; however, as indicated above, we provide a syntax%
\smnote{Saying `their syntax' goes too far, as they don't have
  one. They hint at something that could be be worked out in case of
  working over concrete categories, see p.~176 of their paper.}  by
expressing axioms using the relational signature instead of
necessarily using only equality. We give a sound and complete
deduction system for the arising \emph{relational logic} (which
generalizes standard equational logic), thus obtaining an
explicit description of the monad generated by a relational algebraic
theory in the indicated sense. One consequence of our main result is
that quantitative algebraic theories~\cite{MPP16} induce
$\omega_1$-accessible monads.%
\smnote{I think this is not worked out in the paper; at least I do not
  see where.}  More generally, presentations of $\omega_1$-presentable
monads in our formalism may involve operations with countable
non-discrete arities: indeed, we present an $\omega_1$-ary
relational algebraic theory that defines the metric completion monad.
We also take a glimpse at the more involved setting of
$\kappa$-accessible monads on~$\CatC$ where $\kappa<\lambda$ 
(e.g.~finitary monads on~$\Met$). We give a partial characterization of
$\kappa$-presentable objects in this setting, and show that while the
monad-to-theory direction of our correspondence fails for
$\kappa<\lambda$, the theory-to-monad direction does hold. This 
implies that some salient quantitative algebraic theories induce finitary
monads; e.g. the theory of quantitative join-semilattices~\cite{MPP16}.

\subparagraph*{Related Work}

%
We have already mentioned work by Kelly and Power on finitary
monads~\cite{KP93} and by Mardare et al.~on quantitative algebraic
theories~\cite{MPP16}, as well as our own previous
work~\cite{FordEA21} and our joint work with Ad\'amek~\cite{AFMS20}.

Power and Nishizawa~\cite{NP09} have extended the approach of Kelly
and Power to deal with different enrichments of a category and the
monads thereon, and obtain a correspondence between enriched 
Lawvere theories~\cite{Power99} and finitary enriched monads. More 
recently, Power and Garner~\cite{PG18} have provided a more thorough
understanding of the equivalence between enriched finitary monads and
enriched Lawvere theories as an instance of a free completion of an
enriched category under a class of absolute colimits.
Rosick\'y~\cite{Rosicky21} establishes a monad-theory correspondence
for $\lambda$-accessible enriched monads and a notion of $\lambda$-ary 
enriched theory \emph{\'{a} la} Linton~\cite{Linton69}, where arities of 
operations are given by pairs of objects. Like in the setting of Kelly and 
Power, relations (inequations, distances) are encoded in the arities. So the
syntactic notion of theory is different from (and more abstract than)
ours. Lucyshyn-Wright~\cite{LucyshynWright16} establishes a rather
general correspondence between monads and abstract theories in
symmetric monoidal closed categories, parametric in a choice of
arities, which covers several notions of theory and correspondences in
the categorical literature under one roof.

Kurz and Velebil~\cite{KV17} characterize classical ordered
varieties~\cite{Bloom76} (which are phrased in terms of inequalities)
as precisely the exact categories in an enriched sense with a
`suitable' generator. In recent subsequent work, Ad\'amek at
al.~\cite{ADV21} establish a correspondence of these varieties with
enriched monads on $\Pos$ that are strongly finitary~\cite{KL93},
i.e.~their underlying functor is a left Kan-extension of the embedding
of finite discrete posets into $\Pos$.

The main distinguishing feature of our work relative to the above is
the explicit syntactic description of the monad obtained from a theory
via a sound and complete derivation system.

\section{Preliminaries}

We review the basic theory of locally presentable categories
(see~\cite{AR94} for more detail) and of monads. We assume a modest
familiarity with the elementary concepts of category
theory~\cite{AHS90} and with ordinal and cardinal
numbers~\cite{Jech03}. We write $\card X$ for the cardinality of a set
$X$ and, where $\kappa$ is a cardinal, we write
$X'\subseteq_{\kappa} X$ to indicate that $X'\subseteq X$ and
$\card X<\kappa$. 

\subparagraph*{Locally Presentable Categories}

Fix a regular cardinal $\lambda$ (i.e.~an infinite cardinal which is
not cofinal to any smaller cardinal). A poset $(I, \leq)$ is
$\lambda$-\emph{directed} if each subset $I_0\subseteq_{\lambda} I$
has an upper bound: there exists $u\in I$ such that $i\leq u$ for
all $i\in I_0$. A $\lambda$-\emph{directed diagram} is a
functor whose domain is a $\lambda$-directed poset (viewed as a
category); colimits of such diagrams are also called
\emph{$\lambda$-directed}. An object $X$ in a category $\CatC$ is
\emph{$\lambda$-presentable} if the covariant hom-functor
$\CatC(X, -)$ preserves $\lambda$-directed colimits.  That is, $X$ is
$\lambda$-presentable if for each directed colimit
$(D_i\xrightarrow{c_i} C)_{i\in I}$ in $\CatC$, every morphism
$m\colon X\to C$ factors through one of the $c_i$ essentially
uniquely: there exists $i\in I$ and $g\colon X\to D_i$ such that
$m=c_i\cdot g$, and for all $g'\colon X\to D_i$ such that
$m=c_i\cdot g',$ there exists $j\geq i$ such that
$D(i\to j)\cdot g=D(i\to j)\cdot g'$.

\begin{defn}
  A category $\CatC$ is \emph{locally $\lambda$-presentable} if it is
  cocomplete, its full subcategory $\Pres_{\lambda}(\CatC)$ given by
  the $\lambda$-presentable objects of $\CatC$ is essentially small,
  and every $C\in\CatC$ is a $\lambda$-directed colimit of objects in
  $\Pres_{\lambda}(\CatC).$ When $\lambda=\omega$ (resp.~$\omega_1$),
  we speak of \emph{locally finitely (resp.~countably) presentable
  categories}. We call $\CatC$ \emph{locally presentable} if it is
  locally $\lambda$-presentable for some cardinal $\lambda$.  A
  functor $F$ on a locally presentable category is
  \emph{$\lambda$-accessible} if it preserves $\lambda$-directed
  colimits. When $\lambda=\omega$ or $\omega_1$, we speak of
  \emph{finitary} and \emph{countably accessible functors}, 
  respectively.
\end{defn}

\subparagraph*{Reflective subcategories} A full subcategory~$\CatC'$
of a category~$\CatC$ is \emph{reflective} if the embedding
$\iota\colon\CatC'\hookrightarrow\CatC$ is a right adjoint. In this
case, we write $r_X\colon X\to\Refl X$ (or just~$r$ if~$X$ is clear
from the context) for the universal arrows; we call $\Refl X$
\emph{the reflection of $X\in\CatC$},~$r_X$ the \emph{reflective
arrow}, and the left adjoint~$\Refl$ the \emph{reflector}. The 
universal property of  $r_X\colon X\to\Refl X$ is as follows:
For each morphism $f\colon X\to Y$ in $\CatC$ where $Y$ lies in
$\CatC'$, there exists a unique morphism $f^\sharp\colon\Refl X\to Y$
such that $f=f^\sharp\cdot r_X$.  We call $\CatC'$
\emph{epi-reflective} if $r_X$ is epi for all $X\in\CatC$. We
will employ the following reflection theorem:

\removeThmBraces
\begin{thm}[{\cite[Cor.~2.48]{AR94}}]\label{T:reflection}%
  If~$\CatC'$ is a full subcategory of a locally $\lambda$-presentable
  category~$\CatC$ and~$\CatC'$ is closed under limits and
  $\lambda$-directed colimits in~$\CatC$, then~$\CatC'$ is reflective
  and locally $\lambda$-presentable.
\end{thm}
\resetCurThmBraces

\subparagraph*{Monads}

A \emph{monad} on a category $\CatC$ is a functor
$T\colon\CatC\to\CatC$ equipped with natural transformations
$\eta\colon \Id\to T$ (the \emph{unit}) and $\mu\colon TT\to T$ (the
\emph{multiplication}) such that the diagrams below commute.
\[
\begin{tikzcd}
T \arrow[rd, "\id"'] \arrow[r, "T\eta"] & TT \arrow[d, "\mu"] 
& T \arrow[ld, "\id"] \arrow[l, "\eta T"'] &  
& TTT \arrow[rr, "T\mu"] \arrow[d, "\mu T"'] &  & TT \arrow[d, "\mu"] \\
& T                    &            &  & TT \arrow[rr, "\mu"]        &  & T                   
\end{tikzcd}
\]
We call the monad $(T, \eta, \mu)$ $\lambda$-\emph{accessible}
if its underlying functor is $\lambda$-accessible.
\takeout{        
We denote the category of monads on the category $\CatC$ by
$\Mnd(\CatC)$. Whenever $\CatC$ is locally presentable, we use
$\Mnd_{\kappa}(\CatC)$ to denote the full subcategory of $\Mnd(\CatC)$
given by the $\kappa$-accessible monads on $\CatC$, i.e.  those monads
on $\CatC$ whose underlying functor is $\kappa$-accessible.
}

\begin{defn}
An \emph{Eilenberg-Moore algebra} for the monad
$(T, \eta, \mu)$ is a $\CatC$-morphism of the shape 
$a\colon TX\to X$ satisfying the following coherence 
laws:
\[
\begin{tikzcd}
X \arrow[rd, "\id"'] \arrow[r, "\eta"] & TX \arrow[d, "a"] &  
& TTX \arrow[d, "Ta"'] \arrow[rr, "\mu"] &  & TX \arrow[d, "a"] \\
& X                 &  & TX \arrow[rr, "a"]                      &  & X                
\end{tikzcd}
\]
A \emph{homomorphism} from $a\colon TX\to X$ to an Eilenberg-Moore
algebra $b\colon TY\to Y$ is a morphism $h\colon X\to Y$ in $\CatC$
such that $h\cdot a=Th\cdot b$. \takeout{We write $\EM(T)$ for the category of
Eilenberg-Moore algebras and their homomorphisms.\cfnote{We never
use this notation, but we do use Eilenberg-Moore algebras in discussion.}}
\end{defn}
\begin{notn}
  For a functor $F\colon\CatC\to\CatC$, we write $\Alg F$ for the
  category of $F$-algebras and homomorphisms, i.e. $\Alg F$ has
  $\CatC$-morphisms of the shape $a\colon FA\to A$ as objects, and 
  a homomorphism $(A, a)\to(B, b)$ is a $\CatC$-morphism 
  $h\colon A\to B$ such that $h\cdot a= b\cdot Fh$. 
\end{notn}

\section{Categories of Relational Structures}\label{sec:relationalstructures}

\noindent As indicated previously, we will study monads over base
categories consisting of (single-sorted) relational
structures. Specifically, we will restrict the relational signature to
be finitary but allow infinitary Horn axioms. We proceed to recall
basic definitions, examples, and results, in particular on closed
structure and (local) presentability. In~\autoref{S:presentableobjects}, 
we present new results on the partial characterization of (internally) 
$\lambda$-presentable objects in cases where the overall local 
presentability index of the category is greater than~$\lambda$.
\begin{defn}
\begin{enumerate}
\item A \emph{relational signature} is a set $\Rels$ of \emph{relation
    symbols} $\alpha,\beta,\dots$ together with a finite \emph{arity}
  $0<\arity(\alpha)\in\omega$ for all $\alpha\in\Rels$. A
  $\Rels$-\emph{edge} in a set $S$ is a pair $\edge=\alpha(f)$ where
  $\alpha\in\Rels$ and $f\colon\arity(\alpha)\to S$ is a function. For
  a map $g\colon S\to Y$, we write $g\cdot\edge=\alpha(g\cdot f)$. We
  extend this notation pointwise to sets~$E$ of edges:
  $g\cdot E=\{g\cdot e\mid e\in E\}$.

\item A \emph{$\Rels$-structure} $X$ consists of an underlying set
  $|X|$ (or just~$X$ when no confusion is likely) and a set $\Edge(X)$
  of $\Rels$-edges in $|X|$. If $\alpha(f)\in\Edge(X)$, we write
  $\alpha_X(f)$ or even $X\models\alpha(f)$.

\item A \emph{relation-preserving map} or briefly a \emph{morphism}
  from $X$ to a $\Rels$-structure $Y$ is a function
  $g\colon |X|\to |Y|$ such that
  $g\cdot\Edge(X)\subseteq\Edge(Y)$. If~$g$ additionally is injective
  and \emph{relation-reflecting}, i.e.~whenever $g\cdot e\in\Edge(Y)$
  for an edge~$e$, then~$e\in\Edge(X)$, then~$g$ is an
  \emph{embedding}. We denote by $\Str(\Rels)$ the category of
  $\Rels$-structures and relation-preserving maps.
\end{enumerate}
\end{defn}

\begin{notn}
  Given an edge $\alpha(f)$ such that $f(i):= x_i$ for all $i\in\arity(\alpha)$, 
  we sometimes write $\alpha(x_1,\dots, x_{\arity(\alpha)})$ (or just 
  $\alpha(x_i)$) instead of $\alpha(f)$. We will pass between these 
  presentations without further mention.
\end{notn}

\noindent 
We are going to carve out full subcategories of $\Str(\Rels)$ by
means of infinitary Horn axioms, whose \textbf{syntax} we recall next.
\begin{defn}
  Let~$\Rels$ be a relational signature, and~$\lambda$ a regular
  cardinal. We fix a set $\Vars$ of variables such that
  $\card(\Vars)=\lambda$. A \emph{$\lambda$-ary Horn formula}
  over~$\Rels$ has the form
  \[
    \Phi\rimpl\psi
  \]
  where $\Phi$ is a set of $\Rels$-edges in $\Vars$ such that 
  $\card\Phi<\lambda$ and $\psi$ is a
  $\Rels\cup\{=\}$-edge in $\Vars$, for a fresh binary relation
  symbol~$=$. In case $\Phi=\{\varphi_1,\dots, \varphi_n\}$
  is finite, we write $\varphi_1,\dots, \varphi_n\rimpl\psi$, and
  if~$\Phi=\emptyset$, then we just write $\rimpl\psi$. A 
  \emph{$\lambda$-ary Horn theory} $\Horn=(\Rels, \RelAx)$ consists 
  of a relational signature~$\Rels$ and a set~$\RelAx$ of $\lambda$-ary 
  Horn formulae over~$\Rels$, the \emph{axioms} of~$\Horn$.
\end{defn} 
\noindent \emph{We fix a $\lambda$-ary Horn theory
  $\Horn=(\Rels,\RelAx)$ for the rest of the paper.}
%
%
We define the \textbf{semantics} of Horn formulae in a
$\Rels$-structure $X$ as follows. 
We denote by $\overline{X}$ the
$\Rels\sqcup\{=\}$-structure obtained from $X$ by putting
$=_{X}:=\{(x, x)\mid x\in X\}$.
A \emph{valuation} is a map $\kappa\colon\Vars\to|X|$. 
We say that~$X$ \emph{satisfies} a Horn formula $\Phi\rimpl\psi$ if
whenever~$\kappa$ is a valuation such that $X\models\kappa\cdot\phi$
for all $\phi\in\Phi$, then $\overline{X}\models\kappa\cdot\psi$. Finally,~$X$ is
a \emph{model} of~$\Horn$, or of~$\RelAx$, if~$X$ satisfies all axioms
of~$\Horn$. The full subcategory of $\Str(\Rels)$ spanned
by the models of~$\RelAx$ is $\Str(\Rels, \RelAx)$ (or $\Str \Horn$).

We have an obvious notion of \emph{derivation} under~$\Horn$
\emph{over} a given set~$Z$ (e.g.~of variables or points in a
structure): We extend~$\Horn$ to $(\Rels\cup\{=\},\bar{\RelAx})$
where~$\bar{\RelAx}$ consists of the axioms in~$\RelAx$ and additional
axioms stating that~$=$ is an equivalence and that all relations
in~$\Rels$ are closed under~$=$ in the obvious sense. Then we have a
single ($\lambda$-ary) derivation rule for application of Horn
axioms~$(\Phi\rimpl\psi)\in\bar{\RelAx}$ over~$Z$:
\begin{equation*}
  \infrule{\kappa\cdot\Phi}{\kappa\cdot\psi}\;(\kappa\colon\Vars\to Z).
\end{equation*}
We say that a set~$E$ of edges over~$Z$ \emph{entails} an edge $\edge$
\emph{over~$Z$} (\emph{under~$\Horn$}) if~$\edge$ is derivable from
edges in~$E$ in this system. In case~$Z=\Vars$ and $\card E<\lambda$,
the expression $E\rimpl e$ is in fact a Horn formula, and we then also
say that~$\Horn$ \emph{entails} $E\rimpl\edge$ if~$E$ entails~$\edge$.

\begin{assn}\label{assn:equality}
  For technical convenience, we assume that the fixed Horn theory
  $\Horn=(\Rels,\RelAx)$ \emph{expresses equality}. That is, there
  exists a set $\Eq(x,y)$ of $\Pi$-edges in variables $x,y$ such
  that~$\Horn$ entails $\Eq(x,y)\rimpl x= y$ as well as $\rimpl\psi$
  for all edges~$\psi\in\Eq(x,x)$ (where we use obvious notation for
  substitution; formally, $\Eq(x,x)=g\cdot\Eq(x,y)$ where
  $g(x)=g(y)=x$). Moreover, we assume that~$\RelAx$ explicitly
  includes the (derivable) formulae
  $\Eq(x_1,y_1)\cup\dots\cup\Eq(x_{\arity(\alpha)},y_{\arity(\alpha)})\cup\{\alpha(x_1,\dots,x_{\arity(\alpha)})\}\rimpl\alpha(y_1,\dots,y_{\arity(\alpha)})$
  saying that all relations~$\alpha\in\Rels$ are closed under $\Eq$
  (implying also that~$\Eq$ is symmetric and
  transitive). This is w.l.o.g.~as we can always extend a given Horn
  theory with an equality predicate axiomatized by the above
  conditions without changing its category of models; indeed we leave
  this predicate implicit in examples whose natural presentation does
  not include it.
\end{assn}
\begin{expl}\label{E:HornCategories}
  We mention some key examples of Horn theories:
  \begin{enumerate}
  \item\label{E:item-relcat-set} The category $\Set$ of sets and
    functions is specified by the trivial Horn theory
    $(\emptyset, \emptyset)$.
  
  \item\label{E:item-relcat-pos} The category $\Pos$ of partially
    ordered sets (posets) and monotone maps is specified by the
    $\omega$-ary Horn theory consisting of a single binary relation
    symbol $\leq$ and the axioms
    \[
      x\leq x;\qquad		
      x\leq y,\, y\leq z\rimpl x\leq z; \qquad\text{and}\qquad 
      x\leq y,\, y\leq x\rimpl x=y.        
    \]
    This theory expresses equality (\autoref{assn:equality}) via
    $\Eq(x,y)=\{x\le y,y\le x\}$.
    
  \item\label{E:item-relcat-met} The theory $\Horn_{\Met}$ of
    \emph{metric spaces} is the $\omega_1$-ary theory consisting of
    binary relation symbols $=_{\epsilon}$ for all $\epsilon\in\Q\cap[0,1]$,
    and the axioms
    \begin{align}
      &\rimpl x =_{0} x   								\tag{\textbf{Refl}} \\
      x =_0 y &\rimpl x=y 										\tag{\textbf{Equal}} \\
      x =_{\epsilon} y &\rimpl y =_{\epsilon} x						\tag{\textbf{Sym}} \\
      \{x =_{\epsilon} y, y =_{\epsilon'} z\} &\rimpl x =_{\epsilon+\epsilon'} z \tag{\textbf{Triang}} \\
      x =_{\epsilon} y &\rimpl x =_{\epsilon+\epsilon'} y  				\tag{\textbf{Up}} \\
      \{x =_{\epsilon'} y\mid \Q_{\ge 0}\owns\epsilon'>\epsilon\} &\rimpl x =_{\epsilon} y   	\tag{\textbf{Arch}}
    \end{align}
    where $\epsilon,\epsilon'$ range over~$\Q\cap[0,1]$ (that is, the
    axioms mentioning $\epsilon,\epsilon'$ are in fact axiom schemes
    representing one axiom for each $\epsilon,\epsilon'$). This theory
    expresses equality via $\Eq(x,y)=\{x=_0y\}$; in fact, even if we
    remove $=_0$, the remaining theory still expresses equality via
    $\Eq(x,y)=\{x=_{1/n}y\mid n>0\}$.  The theory $\Horn_{\Met}$
    specifies the category $\Met$ of $1$-bounded metric spaces and
    non-expansive maps, in the sense that~$\Str(\Horn_\Met)$ and
    $\Met$ are concretely isomorphic: $X\in\Str(\Horn_\Met)$ induces
    the $1$-bounded metric space $(X, d)$ given by
    \(
      d(x, y)=\bigwedge\{\epsilon\mid x=_{\epsilon} y\in\Edge(X)\},
    \)
    and conversely a metric space~$(X,d)$ induces the
    $\Horn_\Met$-model on~$X$ with edges
    $\{x=_\epsilon y\mid x,y\in X, d(x,y)\le\epsilon\}$.

  \item Let~$L$ be a complete lattice (for simplicity), and let
    $L_0\subseteq L$ be meet-dense in~$L$ in the sense
    that $l=\bigwedge\{p\in L_0\mid p\ge l\}$ for each $l\in L$;
    whenever $q\ge\bigwedge P$ for $q\in L_0$ and $P\subseteq L_0$
    such that $\bigwedge P\notin L_0$, then $q\ge p$ for some $p\in P$
    (e.g.~these conditions hold trivially for $L_0=L$).  Further, fix
    $\lambda$ such that $|L_0|<\lambda$. Let $\Horn_L$ be the
    $\lambda$-ary Horn theory with binary relation symbols $\alpha_p$
    for all $p\in L_0$ and axioms\cfnote{Does this theory express
      equality?}
    \begin{align*}
      \{\alpha_{p}(x,y)\mid p\in P\}&\rimpl \alpha_{q}(x,y)
      && (P\subseteq L_0,q=\textstyle\bigwedge P\in L_0)\tag{\textbf{Arch}}\\
      \alpha_p(x,y) & \rimpl \alpha_q(x,y) && (p,q\in L_0, p\le q) \tag{\textbf{Up}}
    \end{align*}
    where $p,q$ range over~$L_0$. Then $\Str(\Horn_L)$ is concretely
    isomorphic to the category of \emph{$L$-valued relations}, whose
    objects~$X$ are sets~$X$ equipped with map
    $P\colon X\times X\to L$, and whose morphisms $(X,P)\to (Y,Q)$ are
    maps $X\to Y$ such that $Q(f(x),f(y))\le P(x,y)$. (Of course, the
    previous example is essentially the special case $L=[0,1]$,
    $L_0=\Q\cap[0,1]$ with some additional axioms.)
  \item 
    A signature of partial operations is a set~$P$ of operation
    symbols $f$ with assigned finite arities $\arity(f)$.  A (partial)
    $P$-algebra is then a set $A$ and, for each $f\in P$, a partial
    function $f_A\colon A^{\arity(f)}\to A.$ A homomorphism of partial
    algebras is a map $h\colon A\to B$ such that whenever
    $f_A(x_1,\dots,x_{\arity(f)})$ is defined, then
    $f_B(h(x_1),\dots,h(x_{\arity(f)}))$ is defined and equals
    $h(f_A(x_1,\dots,x_{\arity(f)}))$. The category of partial
    $P$-algebras and their homomorphims is concretely isomorphic to
    the category of models of the $\omega$-ary Horn theory consisting
    of relational symbols $\alpha_f$ of arity $\arity(f)+1$ for all
    $f\in P$ (with $\alpha_f(x_1,\dots, x_{\arity(f)}, y)$ being
    understood as $f(x_1,\dots, x_{\arity(f)})=y$), and axioms
    \[
      \{\alpha(x_1,\dots, x_{\arity(f)}, y), \alpha(x_1,\dots, x_{\arity(f)}, z)\}\rimpl y= z.
    \]
  \end{enumerate}
\end{expl}
We proceed to recall some key aspects of the categorical structure of
$\Str(\Horn)$.

\subparagraph*{Reflection} We first note
\begin{propn}\label{prop:refl}
  $\Str(\Rels, \RelAx)$ is a (full) epi-reflective  subcategory of
  $\Str(\Rels)$.
\end{propn}
\noindent Since $\Str(\Rels)$ is easily seen to be complete and
cocomplete, it follows that $\Str(\Rels,\RelAx)$ is cocomplete and
moreover closed under limits in $\Str(\Rels)$, and hence complete. We
write
\begin{equation*}
  \Refl\colon\Str(\Rels)\to\Str(\Rels,\RelAx)\quad\text{and}\quad r_X\colon X\to\Refl X
\end{equation*}
for the left adjoint of the inclusion
$\Str(\Rels,\RelAx)\subto\Str(\Rels)$ (the \emph{reflector}) and the
corresponding (surjective) reflection maps, respectively. Explicitly,
$\Refl X$ is constructed as follows. We define an equivalence~$\sim$
on~$X$ by $x\sim y$ if $\Edge(X)$ entails $x=y$ under~$\Horn$ (in the
sense defined above), and let $q\colon X\to X/\mathord\sim$ denote the
quotient map; then $\Refl X$ has underlying set~$X/\mathord\sim$, and contains
precisely the edges $q\cdot\edge$ such that $\Edge(X)$
entails~$\edge$; moreover, $r_X=q$ as a map.

\subparagraph*{Local presentability} One easily checks
\begin{lem}\label{prop:fp-rels}
  An object $(X,E)\in\Str(\Rels)$ is $\lambda$-presentable iff
  $\card X<\lambda$ and~$\card E<\lambda$; the category $\Str(\Rels)$
  is locally finitely presentable.
\end{lem}
By \autoref{prop:refl} and since $\Str(\Rels,\RelAx)$ is easily seen
to be closed under $\lambda$-directed colimits in~$\Str(\Rels)$, we
thus have \removeThmBraces
\begin{propn}[{\cite[Example 5.27(3)]{AR94}}]\label{P:presentable}
  $\Str(\Horn)$ is locally $\lambda$-presentable.
\end{propn}
\resetCurThmBraces
\noindent The forgetful functor $\Str(\Horn)\to\Set$ preserves
$\lambda$-directed colimits. Moreover, we have an easy
characterization of $\lambda$-presentable objects:
\begin{propn}\label{prop:pres-object}\label{R:RelsStructures-itm-fp}
  For an $\Horn$-model $X$, the following are equivalent.
  \begin{enumerate}
  \item\label{item:pres} $X$ is $\lambda$-presentable in $\Str(\Rels, \RelAx)$;
  \item\label{item:refl-pres} $X\cong\Refl(Y,E)$ for some $\lambda$-presentable
    $(Y,E)\in\Str(\Rels)$;
  \item\label{item:generated} $\card|X|<\lambda$, and $X$ is
    \emph{$\lambda$-generated}, i.e.~there exists $E\subseteq\Edge(X)$
    such that $\card E<\lambda$ and $E$ entails every edge
    in~$\Edge(X)$ under~$\Horn$ (equivalently, $\Refl i$ is an
    isomorphism where $i\colon (|X|,E)\to X$ is the
    $\Str(\Rels)$-morphism carried by~$\id_X$).
\end{enumerate}
\end{propn}
\begin{rem}
  For instance, every finite partial order is $\omega$-presentable,
  and every countable metric space is $\omega_1$-presentable.  We
  emphasize that the situation is more complicated for
  $\kappa$-presentable objects where $\kappa<\lambda$; we treat this case in
  more detail in \autoref{S:presentableobjects}. For instance, every
  finite metric space with rational distances
  (cf.~\autoref{E:HornCategories}) is finitely generated in the sense
  of \autoref{prop:pres-object} but not finitely presentable.
\end{rem}

%
    


\subparagraph*{Closed monoidal structure} The pointwise structure
defines an \emph{internal hom} functor:
\begin{defn}\label{D:internalhom}
  The \emph{internal hom of $X, Y\in\Str(\Rels)$} is the
  $\Rels$-structure $[X, Y]$ carried by $\Str(\Rels)(X, Y)$ with set
  of edges
  \[
    \Edge([X, Y]):=\{e \mid \forall x\in X.\, \pi_x\cdot
    e\in\Edge(Y)\}
  \]
  where $\pi_x\colon \Str(\Rels)(X, Y)\to Y$ is defined by
  $\pi_x(g)= g(x)$. For each $X\in\Str(\Rels)$, the assignment
  $Y\mapsto [X, Y]$ defines a (covariant) \emph{internal hom functor}
  \[
  [X, -]\colon\Str(\Rels)\to\Str(\Rels)
  \]
  with the action on a morphism $m\colon Y\to Z$ 
  given by post-composition: $[X, m](g):= m\cdot g$.
\end{defn}
\noindent One easily checks that $[X,-]$ restricts to
\[[X,-]\colon\Str(\Horn)\to\Str(\Horn).\]
It turns out that the internal hom functor is always part of a closed
symmetric monoidal structure on $\Str(\Horn)$. Indeed, for
$X,Y\in\Str(\Rels)$, we define $X\otimes Y$ as the structure with
underlying set $X\times Y$ and edges
\[
  \{e\mid (\pi_1\cdot e\text{ constant }\land \pi_2\cdot e\in\Edge(Y))
  \lor (\pi_2\cdot e\text{ constant }\land \pi_1\cdot e\in\Edge(X))\}.
\]
where an edge $(\alpha,f)$ is \emph{constant} if~$f$ is a constant
map, and $\pi_1\colon X\times Y\to X$ and $\pi_2\colon X\times Y\to Y$
are the projection maps. We then have $\Str(\Pi)$-morphisms
$u_Y\colon Y\to[X,Y\otimes X]$, $u_Y(y)(x)=(y,x)$.
It is straightforward to check that $u_Y$ is a universal arrow. That
is:
\begin{propn}\label{P:closedmonoidal}
  For every $X\in\Str(\Rels)$, $(-)\otimes X$ is a left adjoint of
  $[X, -]$.
\end{propn}
\begin{cor}\label{c:tensor-refl}
  For every $X\in\Str(\Horn)$, $\Refl((-)\otimes X)$ is a left adjoint
  of $[X, -]\colon\Str(\Horn)\to \Str(\Horn)$.
\end{cor}
That is, $\Str(\Horn)$ is a closed symmetric monoidal category, with
monoidal structure \[X\otimes_\Horn Y = \Refl(X\otimes Y).\] We
briefly refer to $\otimes_\Horn$ as the \emph{Manhattan product}.

\begin{expl}\label{E:Manhattan}
  \begin{enumerate}
  \item In $\Pos$ (\autoref{E:HornCategories}\ref{E:item-relcat-pos}),
    the Manhattan product coincides with binary Cartesian product
    (so~$\Pos$ is Cartesian closed).
  \item In $\Met$ (\autoref{E:HornCategories}\ref{E:item-relcat-met}),
    the Manhattan product $(X,d_X)\otimes_{\Horn_\Met}(Y,d_Y)$ is
    $X\times Y$ equipped with the well-known Manhattan metric~$d$
    given by
    $d((x_1,y_1),(x_2,y_2))=\min(d_X(x_1,x_2)+d_Y(y_1,y_2),1)$ (while
    Cartesian products carry the supremum metric).
  \end{enumerate}
\end{expl}

\begin{defn}
A functor $F\colon\Str(\Horn)\to\Str(\Horn)$ that preserves the pointwise 
structure on morphisms is called \emph{enriched}. That is, 
we call $F$ enriched if for all $X, Y\in\Str(\Horn)$ and all edges 
$f\colon\arity(\alpha)\to\Str(\Horn)(X, Y)$~$(\alpha\in\Rels)$, if 
$[X, Y]\models\alpha(f_i)$, then $[FX, FY]\models\alpha(F(f_i))$.
\end{defn}

\subparagraph*{Internal local presentability} For use of objects~$X$
as arities of operations, we will in fact need that the
\emph{internal} hom $[X,-]$ is $\lambda$-accessible, in which case we
say that~$X$ is \emph{internally $\lambda$-presentable}. For
distinction, we sometimes say that~$X$ is \emph{externally
$\lambda$-presentable} if~$X$ is $\lambda$-presentable in the
standard sense. Indeed, one can show using
\autoref{prop:pres-object} that the class of $\lambda$-presentable objects
in $\Str(\Horn)$ is closed under the Manhattan tensor; it follows
fairly straightforwardly that
\begin{propn}\label{prop:internally-pres}
  Every $\lambda$-presentable object in $\Str(\Horn)$ is internally
  $\lambda$-presentable.
\end{propn}
\noindent Indeed this implies that $\Str(\Horn)$ is \emph{internally
  locally $\lambda$-presentable}~\cite{Kelly82b}.
\subsection{Compact Horn Models}\label{S:presentableobjects}

We have seen above that the category $\Str(\Horn)$ (where $\Horn$ is a
$\lambda$-ary Horn theory) is (internally) locally
$\lambda$-presentable, with a straightforward characterization of the
(internally) $\lambda$-presentable objects
(Propositions~\ref{prop:pres-object} and~\ref{prop:internally-pres}). We
proceed to look at the rather less straightforward notion of
internally $\kappa$-presentable objects in $\Str(\Rels, \RelAx)$ for
$\kappa <\lambda$. The main scenario that motivates our interest in this
case is that of finitary monads on categories that are internally
locally $\lambda$-presentable only for some~$\lambda>\omega$, such as
metric spaces. 

Further unfolding definitions. we have that an object~$X$ is
internally $\kappa$-presentable if for every
$\kappa$-directed colimit $(D_i\xrightarrow{c_i} C)_{i\in I}$, the
canonical morphism
\begin{equation*}
  \colim [X,D_i]\to[X,\colim D_i]
\end{equation*}
is an isomorphism. We split this property into two parts: We say
that~$X$ is \emph{weakly $\kappa$-presentable} if the canonical morphism
is always surjective, and \emph{co-weakly $\kappa$-presentable} if the
canonical morphism is always an embedding. \lsnote{Actually, that is
  overkill because with embeddings, epimorphisms would suffice, which
  need not be surjective. Maybe we in fact want quotient/injective?}
Below, we give necessary and sufficient conditions for weak
$\kappa$-presentability. Co-weak $\kappa$-presentability is a more elusive
property; more concretely, it means roughly that $X$-indexed tuples of
derivations in the given Horn theory can be synchronized into single
derivations over $X$-indexed tuples of points. We give some examples
below (\autoref{expl:internal-fp}).

%
\begin{expl}\label{expl:internal-fp}
  We give some examples and non-examples of internally finitely
  presentable objects in locally $\omega_1$-presentable categories
  $\Str(\Pi,\RelAx)$.
  \begin{enumerate}
  \item A metric space is internally finitely presentable iff it is
    finite and discrete. The `if' direction has surprisingly
    complicated reasons: It holds only because over the reals, finite
    joins distribute over directed infima. On the other hand, no
    non-empty metric space is \emph{externally} finitely presentable,
    as its hom-functor will fail to preserve the colimit of the
    directed chain $(D_i)_{i<\omega}$ of spaces $D_i$ with underlying
    set $\{0,1\}$ and metric $d(0,1)=1/(i+1)$.
  \item\label{item:non-distrib} In the category of $L$-valued relations
    for a complete lattice~$L$ in which binary joins fail to
    distribute over directed infima (such lattices exist), the
    two-element discrete space fails to be internally finitely
    presentable.
  \item\label{item:non-distrib-fp} Let~$L$ be as in the previous item,
    and assume additionally that there is $l\in L$ such that in the
    downset of~$l$, finite joins do distribute over directed infima
    (again, such~$L$ exist). Take the Horn theory of $L$-valued
    relations, extended with an additional (two-valued)
    relation~$\alpha$ and axioms
    \begin{equation*}
      \alpha(x,y)\land \alpha(x',y')\rimpl x=_lx'\qquad \alpha(x,y)\land \alpha(x',y')\rimpl y=_ly'.
    \end{equation*}
    Then the set $\{0,1\}$ equipped with the discrete $L$-valued
    relation and $\alpha(0,1)$ is internally finitely presentable.
  \end{enumerate}
\end{expl}
\noindent We proceed to give the announced characterization of weakly
finitely presentable objects.

\begin{defn}
  A \emph{cover} $(Y,E)$, or just~$E$, of $X\in\Str(\Rels, \RelAx)$ is
  a set $E$ of edges in some set $Y\supseteq|X|$ such that all edges
  of~$X$ are implied by those in~$E$ under the Horn theory $\RelAx$.
  That is, the underlying map $r_{(Y,E)}\colon Y \to |\Refl (Y,E)|$ of
  the reflection composes with the inclusion $i\colon |X|\into Y$ to
  yield a morphism $r_{(Y,E)}\cdot i\colon X \to \Refl (Y,E)$ (in
  $\Str(\Rels, \RelAx)$). Then~$X$ is $\kappa$-\emph{compact} if for
  each cover $(Y,E)$ of $X$ there exist $E'\subseteq_{\kappa} E$ and
  a morphism $f\colon X\to \Refl(Y,E')$ such that
  $r_{(Y,E)}\cdot i=\Refl j\cdot f$ where $j\colon(Y,E')\to(Y,E)$ is
  the $\Str(\Rels)$-morphism carried by~$\id_Y$:
  \begin{equation}\label{diag:compact}
    \begin{tikzcd}[column sep=4em]
      &  \Refl(Y,E')\arrow[d,"\Refl j"]\\
      X \arrow[ur,"f"] \arrow[r,"r_{(Y,E)}\cdot i" below] & \Refl(Y,E)
    \end{tikzcd}
  \end{equation}
\end{defn}
\begin{lem}\label{lem:compact-generated}
  Every~$\kappa$-compact object is $\kappa$-generated.
\end{lem}
\begin{rem}\label{rem:compact}
  We will show that the weakly finitely presentable objects in
  $\Str(\Rels,\RelAx)$ are precisely the $\kappa$-compact objects
  with less than~$\kappa$ elements (\autoref{P:fpchar}). This
  characterization breaks under seemingly innocuous variations of the
  definition of~$\kappa$-compactness:
  \begin{enumerate}
  \item It is essential that the edges of a cover live over a
    superset~$Y$ of~$X$. If we were to restrict covers to consist of
    edges over~$X$ (call such a cover an \emph{$X$-cover}), then
    finite $\omega$-compact objects in the arising relaxed sense would
    in general fail to be finitely
    presentable. E.g.~take~$(\Rels,\RelAx)$ to be the theory of metric
    spaces additionally equipped with a transitive
    relation~$\alpha$. Then the set~$X=\{0,2\}$ equipped with the
    discrete metric and the edge $\alpha(0, 2)$ satisfies the relaxed
    definition of compactness (every $X$-cover must contain the edge
    $\alpha(0,2)$) but fails to be weakly finitely presentable: The
    colimit of the $\omega$-chain of objects~$D_i$ with underlying set
    $\{0,1,1',2\}$, distances $d(0,1)=d(1',2)=1$, $d(1,1')=1/i$, and
    edges $\alpha(0,1)$ and $\alpha(1',2)$ is not weakly preserved by
    the hom-functor $\Str(\Rels,\RelAx)(X,-)$ (the obvious inclusion of $X$
    into the colimit fails to factorize through any of the~$D_i$).
  \item Note that we do not require that the factorization~$f$ of
    $r_{(Y,E)}\cdot i$ in~\eqref{diag:compact} equals
    $r_{(Y,E')}\cdot i$; i.e.~$f$ may rename elements of~$X$ into
    elements of~$Y$ that lie outside~$X$. Let us refer to the
    natural-sounding strengthening of $\kappa$-compactness where we
    do require $f=r_{(Y,E')}\cdot i$ as \emph{strong}
    $\kappa$-compactness; e.g.~$X$ is strongly $\omega$-compact if
    every cover of~$X$ has a finite subcover. However, this notion is
    too strong, i.e.~not every (weakly) finitely presentable object in
    $\Str(\Rels,\RelAx)$ is strongly $\omega$-compact. As a
    counterexample, consider the same Horn theory as in the previous
    item but without the transitivity axiom for~$\alpha$. Then the
    same object~$X$ as in the previous item is weakly finitely
    presentable (even internally finitely presentable) but not
    strongly $\omega$-compact, as witnessed by the cover
    $E=\{\alpha(0',2)\}\cup\{0=_{1/n}0'\mid n>0\}$.
  \end{enumerate}
\end{rem}
\begin{propn}\label{P:fpchar}
The following are equivalent for $X\in\Str(\Rels, \RelAx)$:
\begin{enumerate}
\item\label{P:fpchar-itm1}
	$X$ is weakly $\kappa$-presentable;
\item\label{P:fpchar-itm2}
	$X$ is $\kappa$-compact, and $\card|X|<\kappa$.
\end{enumerate}
\end{propn}
\section{Relational Algebraic Theories}
\label{sec:algebras}

We next describe a framework of universal algebra for enriched
$\kappa$-accessible monads on the internally locally
$\lambda$-presentable category $\BC=\Str(\Horn)$ of $\Horn$-models,
for $\kappa\le\lambda$. We etablish one direction of our theory-monad
correspondence: We show that every theory in our framework induces a
$\kappa$-accessible monad (\autoref{R:free-monad}) whose algebras are
precisely the models of the theory (\autoref{T:monadic}). We address
the converse direction in \autoref{S:monad-theory}. We write
$\BC_0$ for the ordinary category underlying the closed monoidal
category $\BC$.%

Following Kelly and Power~\cite{KP93}, we use the internally
$\lambda$-presentable objects in $\BC$ as the arities of operation
symbols. The full subcategory $\Pres_{\lambda}(\BC)$ of internally
$\lambda$-presentable objects is essentially small
(\autoref{prop:internally-pres}); \emph{we fix a small subcategory
  $\pres_{\lambda}$ of internally $\lambda$-presentable $\BC$-objects
  representing all such objects up to isomorphism}.  For all infinite
$\kappa<\lambda$, the full subcategory
$\pres_{\kappa}\hookrightarrow\pres_{\lambda}$ is given by the internally
$\kappa$-presentable objects in~$\pres_{\lambda}$. 
\begin{defn}
  Let $\kappa\leq\lambda$ be a regular cardinal.  A \emph{$\kappa$-ary
    signature} is a set~$\Sigma$ of \emph{operation symbols}~$\sigma$,
  each of which is equipped with an \emph{arity}
  $\arity(\sigma)\in\pres_{\kappa}$.

  A \emph{$\Sigma$-algebra} $A$ consists of a $\BC$-object, also
  denoted $A$, and a family of $\BC$-morphisms
  \[
    \sigma_A\colon [\arity(\sigma), A]\to A \qquad (\sigma\in\Sigma)
  \]
  A \emph{homomorphism} from $A$ to a $\Sigma$-algebra $B$ is a
  morphism $h\colon A\to B$ in $\BC$ such that the diagram below
  commutes for all $\sigma\in\Sigma$.
  \[
    \begin{tikzcd}[column sep = 40]
      {[\arity(\sigma), A]} \arrow[r, "\sigma_A"] \arrow[d,  "h\cdot(-)"']
      &  
      A \arrow[d, "h"]
      \\
      {[\arity(\sigma), B]}
      \arrow[r, "\sigma_B"]
      &
      B               
    \end{tikzcd}
  \]
  \noindent We write $\AlgSigma$ for the category of $\Sigma$-algebras
  and homomorphisms. By a \emph{subalgebra} of the $\Sigma$-algebra
  $A$, we understand a $\Sigma$-algebra $B$ equipped with a
  homomorphism $h\colon B\hookrightarrow A$ whose underlying
  $\BC$-morphism is an embedding. 
\end{defn}

\subparagraph*{Signatures and their algebras} \emph{Fix a $\kappa$-ary signature
  $\Sigma$ for the remainder of this section}. The category
$\AlgSigma$ can be presented as a category of functor algebras:

\begin{defn}
The \emph{signature functor associated to $\Sigma$},  
$H_{\Sigma}\colon\BC\to\BC$, is given by 
\[\textstyle
H_{\Sigma}=\coprod_{\sigma\in\Sigma}[\arity(\sigma), -].
\]
\end{defn}
\noindent 
The categories $\AlgSigma$ and $\Alg H_{\Sigma}$ are clearly
isomorphic as concrete categories over $\BC$, so the forgetful functor
$\AlgSigma\to\BC_0$ inherits all properties of the forgetful functor
$\Alg H_\Sigma\to\BC_0$. We collect a few basic consequences of this
observation:
\begin{rem}\label{rem:alg-sigma-props}
  \begin{enumerate}
  \item\label{item:alg-sigma-limits} In general, the forgetful functor
    $\mathcal{U}\colon\Alg F\to\CatC$ from the category $\Alg F$ of
    $F$-coalgebras for a functor $F$ on a category $\CatC$ creates all
    limits in $\CatC$. 
    It follows that $\AlgSigma$ has all limits, and the forgetful
    functor $\AlgSigma\to\BC_0$ creates them.
	
    
  \item\label{item:alg-sigma-accessble} Since $H_\Sigma$ is a colimit
    of $\kappa$-accessible functors $[\arity(\sigma),-]$, it is itself
    $\kappa$-accessible, so that the forgetful functor
    $\Alg H_\Sigma\to\BC_0$ creates $\kappa$-directed colimits, and
    the same holds for the forgetful functor $\AlgSigma\to\BC_0$.
  \item From the previous observation (which implies that~$H_\Sigma$
    is also $\lambda$-accessible) and \autoref{P:presentable}, we
    obtain by~\cite[Remark~2.75]{AR94} (for $\lambda$-accessible
    functors~$F$ on locally $\lambda$-presentable categories, $\Alg F$ is
    locally $\lambda$-presentable) that \emph{$\AlgSigma$ is locally
      $\lambda$-presentable}.
  \item\label{item:alg-sigma-adjoint} Ad\'amek~\cite{Adamek74} shows
    that for a $\lambda$-accessible functor~$F$ on a cocomplete
    category~$\CatC$, the forgetful functor $\Alg F\to\CatC$ is right
    adjoint. From~\ref{item:alg-sigma-accessble} and cocompleteness
    of~$\BC_0$ (\autoref{sec:relationalstructures}), we thus
    obtain 
    that the forgetful functor $\AlgSigma\to\BC_0$ is right adjoint;
    that is, \emph{every object $X\in\BC$ generates a free
      $\Sigma$-algebra $F_{\Sigma} X$}.
\end{enumerate}
\end{rem}


\subparagraph*{Varieties of $\Sigma$-Algebras}
We now describe a syntax for specifying full subcategories 
of $\AlgSigma$. As a first step, we introduce a notion of 
$\Sigma$-term, defined as usual in universal algebra: 
\begin{defn}[$\Sigma$-Terms; substitution]\label{def:subst}
  For $X\in\BC$, we call its underlying set $|X|$ the 
  set of \emph{variables} in $X$. The set $\TSigma(X)$ 
  of \emph{$\Sigma$-terms in $X$} is defined inductively 
  as follows:
  \begin{enumerate}
  \item Each variable in~$|X|$ is a $\Sigma$-term in $X$;
  \item For each $\sigma\in\Sigma$ and each map
    $f\colon|\arity(\sigma)|\to T_\Sigma(X)$, $\sigma(f)$ is a 
    $\Sigma$-term in $X$.
  \end{enumerate}
  We usually omit the signature $\Sigma$ from the notation and 
  speak simply of \emph{terms} (in $X$). We employ standard 
  syntactic notions: A \emph{substitution} is a map
  $\tau\colon |Y|\to\TSigma(X)$, for $X,Y\in\BC$. We extend~$\tau$ to
  a map~$\bar\tau$ on terms $t\in\TSigma(X)$ as usual. Formally, we
  define $\bar\tau(t)$ inductively by $\bar\tau(x)=\tau(x)$ for
  $x\in X$, and $\bar\tau\sigma(f) = \sigma(\bar\tau\cdot f)$ for
  $f\colon\arity(\sigma)\to\TSigma(X)$. We will not further
  distinguish between~$\tau$ and~$\bar\tau$ in the notation, writing
  $\tau(t)=\bar\tau(t)$ and $\tau\cdot f=\bar\tau\cdot f$ for~$t,f$ as
  above. Moreover, the set $\sub(t)$ of \emph{subterms} of a term
  $t\in\TSigma(X)$ is defined as usual; formally, we simultaneously
  define $\sub(t)$ and $\sub(f)$ for $f\colon I\to\TSigma(X)$
  (with~$I$ some index set or object) inductively by $\sub(x)=\{x\}$
  for $x\in X$; $\sub(\sigma(f))=\{\sigma(f)\}\cup\sub(f)$ for
  $f\colon\arity(\sigma)\to\TSigma(X)$; 
  and $\sub(f)=\bigcup_{i\in I}\sub(f(i))$.
\end{defn}

\noindent Note that term formation operates without regard for the
relational structure. Consequently, the evaluation of terms in a given
$\Sigma$-algebra may fail to be defined:
\begin{defn}\label{D:evaluation}
  Let $A$ be a $\Sigma$-algebra. For a an object $X\in\BC$ 
  and a relation-preserving assignment $e\colon X\to A$, the
  partial \emph{evaluation map}
  $e^{\#}\colon T_{\Sigma}(X)\to A$ is inductively defined by
  \begin{enumerate}
  \item $e^{\#}(x)= e(x)$ for $x\in X$, and
    
  \item $e^{\#}(\sigma(f))$ is defined for $\sigma\in\Sigma$ and
    $f\colon|\arity(\sigma)|\to\TSigma(X)$ iff the following hold:
    \begin{enumerate}
    \item $e^{\#}\cdot f(i)$ is defined for all $i\in\arity(\sigma)$,
      and
      
    \item if $\alpha(g)$ is a $\Rels$-edge in $\arity(\sigma)$, then
      $A\models\alpha(e^{\#}\cdot (f\cdot g))$.
    \end{enumerate}
    In case $e^{\#}(\sigma(f))$ is defined, we put
    $e^{\#}(\sigma(f))= \sigma_A(e^{\#}\cdot f)$.
  \end{enumerate}
\end{defn}
\noindent As indicated previously, we phrase theories using the
relations in~$\Rels$:
\begin{defn}\label{D:relational-theory}
  A \emph{$\Sigma$-relation} $X\vdash\alpha(f)$ consists of a
  \emph{context} $X\in\BC$ and a $\Rels$-edge $\alpha(f)$ in
  $\TSigma(X)$. We say that $X\vdash\alpha(f)$ is \emph{$\kappa$-ary}
  if $X\in\pres_\kappa$.
  A $\Sigma$-algebra $A$ \emph{satisfies}
  $X\vdash\alpha(f)$ if, for each relation preserving
  assignment $e\colon X\to A$, $e^{\#}\cdot f(i)$ is defined 
  for all $i\in X$, and $\alpha_A(e^{\#}\cdot f)$.
  A (\emph{$\kappa$-ary}) \emph{relational algebraic theory}
  $(\Sigma,\E)$ consists of the ($\kappa$-ary) signature~$\Sigma$ and
  a set~$\E$ of $\kappa$-ary $\Sigma$-relations.
  It determines the
  subcategory $\Alg(\Sigma, \E)$ of $\AlgSigma$ consisting of those
  $\Sigma$-algebras which satisfy each $\Sigma$-relation in~$\E$. We
  refer to categories of the shape $\Alg(\Sigma, \E)$ as
  \emph{varieties} of $\Sigma$-algebras.
\end{defn}
\begin{rem}
  For $\BC_0=\Pos$, the above notion of variety of $\Sigma$-algebras
  corresponds precisely to what we have termed `varieties of coherent
  algebras' in earlier work with
  Ad\'{a}mek~\cite{AFMS20}. 
\end{rem}

\begin{expl}\label{expl:cmet}
  Recall that a (1-bounded) metric space $X$ is \emph{complete} if
  every Cauchy sequence $(x_i)_{i\in\omega}$ of points in $X$ has a
  limit in $X$. That is, if $(x_i)$ satisfies the Cauchy property%
  \begin{equation}\label{eq:cauchy}
    \forall\epsilon>0.~\exists N_{\epsilon}\in\omega.~\forall n,m\ge
    N_{\epsilon}~(d(y_n, y_m)<\epsilon),
  \end{equation}
  then there is a point $\Lim(x_i)\in X$ with the property of a limit:
  for all $\epsilon>0$ there is $N\in\omega$ such that
  $x_n =_{\epsilon} \Lim(x_i)$ for all $n\geq N$. The full subcategory
  $\CMS\hookrightarrow\Met$ of complete metric spaces is specified by
  the relational algebraic theory described below. Thus, by \autoref{T:monadic} 
  below, we recover the fact that $\CMS$ is monadic over $\Met$. Furthermore, 
  we obtain a completely syntactic $\omega_1$-ary description of the metric 
  completion monad via the deduction system introduced later in this section.

  The \emph{theory $\mathbb{T}_{\CMS}$ of complete metric spaces} has
  $\Gamma$-ary \emph{limit operations} $\Lim_{\Gamma}$ for all
  spaces $\Gamma\in\pres_{\omega_1}$ of the form $\Gamma= \{x_i\mid i\in\omega\}$
  where $(x_i)_{i\in\omega}$ is a Cauchy sequence in~$\Gamma$. The
  axioms of $\mathbb{T}_{\CMS}$ then say precisely that
  $\Lim_{\Gamma}(x_i)$ is a limit of $(x_i)$. Explicitly, for all
  $\Gamma$ as above, we impose all axioms of the shape
  \[
    \Gamma\vdash\Lim_{\Gamma}(x_n) =_{\epsilon} x_k \quad (k\geq N_{\epsilon})
    \qquad\qquad\text{where~$N_\epsilon$ is as in~\eqref{eq:cauchy}}.
  \]
\end{expl}

\noindent \emph{We fix a variety $\V=\Alg(\Sigma, \E)$ for the
  remainder of this section.} We are going to see that $\V$ is a
reflective subcategory of $\AlgSigma$ by application of
\autoref{T:reflection}, i.e.~we show that~$\V$ is closed under limits
and $\kappa$-directed colimits in $\AlgSigma$. We state the second
property separately:

\begin{propn}\label{P:creation}
$\V$ is closed under $\kappa$-directed colimits in $\AlgSigma$.
\end{propn}
\noindent Combining this with
\autoref{rem:alg-sigma-props}\ref{item:alg-sigma-accessble}, we obtain
\begin{cor}\label{cor:variety-accessible}
  The forgetful functor $V\colon\V\to\BC$ is $\kappa$-accessible.
\end{cor}

\noindent 
It is fairly straightforward to show
that~$\V$ is also closed under products and subobjects, and hence
under limits (in $\AlgSigma$). Thus, as announced, we have: 
\begin{propn}\label{P:varieties}
 $\V$ is a reflective subcategory of $\AlgSigma$.
\end{propn}

\begin{rem}\label{R:free-monad}
  It follows that the forgetful functor $\V\to\BC_0$ has a left
  adjoint, namely the composite
  $\BC\xrightarrow{F_{\Sigma}}\AlgSigma\xrightarrow{R_{\V}}\V$, where
  $F_\Sigma$ is the left adjoint of the forgetful functor
  $\AlgSigma\to\BC$~(\autoref{rem:alg-sigma-props}\ref{item:alg-sigma-adjoint})
  and $R_{\V}$ is the reflector according to \autoref{P:varieties}.
  We call the ensuing monad~$\bbT_{\V}$ the \emph{free-algebra monad}
  of $\V$; by \autoref{cor:variety-accessible},~$\bbT_{\V}$ is
  $\kappa$-accessible.
\end{rem}
\noindent Indeed,~$\V$ is essentially the category of Eilenberg-Moore
algebras of~$\bbT_{\V}$: Using Beck's monadicity theorem, one can show
that
\begin{thm}\label{T:monadic}
  The forgetful functor $\V\to\BC_0$ is monadic.
\end{thm}
%
\begin{cor}\label{C:monadic}
  Every $\kappa$-ary relational algebraic theory
  may be translated into an enriched $\kappa$-accessible monad, preserving
  categories of models.
\end{cor}

%
%
\subparagraph*{Relational Logic}\label{S:relational-logic}%
We proceed to set up a system of rules for deriving relations among
terms. The calculus will involve two forms of judgements, both
mentioning a context~$X\in\Str(\Rels)$ (not necessarily
$\kappa$-presentable). By a \emph{relational judgement}%
\begin{equation*}
  X\vdash \alpha(t_1,\dots,t_{\arity(\alpha)}),
\end{equation*}
where $t_1,\dots,t_{\arity(\alpha)}\in \TSigma(X)$, we indicate that
for every valuation of~$X$ that is \emph{admissible}, i.e.~satisfies
the relational constraints specified in~$X$, the terms $t_i$ are
defined, and the resulting tuple of values is in relation~$\alpha$. We
treat expressions $\alpha(t_1,\dots,t_{\arity(\alpha)})$ notationally
as edges over~$\TSigma(X)$, in particular sometimes write them in the
form $\alpha(f)$ for $f\colon\arity(\alpha)\to\TSigma(X)$. Moreover,
a \emph{definedness judgement} of the form
\begin{equation*}
  X\vdash\isdef t
\end{equation*}
states that~$t$ is defined for all admissible valuations of~$X$. (We
could encode $\isdef t$ as $\phi(t,t)$ for any $\phi\in Eq(x,y)$ but
for technical reasons we prefer to keep definedness judgement distinct
from relational judgements.) 

The rules of the arising system of \emph{relational 
logic} are shown below:
%
\begin{gather*}
   (\mathsf{Var})\; \frac{}{X\vdash\isdef x \quad}~(x\in X) 
   \qquad
  (\mathsf{Ctx})\; 
  \frac{}{X\vdash\alpha(x_1,\dots,x_{\arity(\alpha)})}~
  (X\models\alpha(x_1,\dots,x_{\arity(\alpha)}))
  \\[3mm]
  (\mathsf{Mor}) \; 
  \frac{
  	\{X\vdash\alpha(f_i(j))\mid
        j\in\arity(\sigma)\}\cup\{X\vdash\isdef\sigma(f_i)\mid i\in\arity(\alpha)\}
        }{X\vdash\alpha(\sigma(f_i))}~((f_i\colon\arity(\sigma)\to\TSigma(X))_{i\in\arity(\alpha)})
  \\[3mm]
(\EAr)
  \;  
  \frac{
      \{X\vdash\alpha(f\cdot g)\mid 
      \alpha(g)\in\arity(\sigma)\}\,
      \cup\,\{X\vdash\isdef f(i)\mid i\in\arity(\sigma)\}
    }
  {X\vdash\isdef\sigma(f)}~ (f\colon\arity(\sigma)\to\TSigma(X))
  \\[3mm]
  (\IAr) \;  \frac{\{X\vdash\alpha(\tau\cdot f)\mid \alpha(f)\in\Delta\}\,
  \cup\,\{X\vdash\isdef \tau(y)\mid y\in\Delta\}}								
  {X\vdash\beta(c)}~(+) \\
  %
  (\mathsf{RelAx}) \;
  \frac{\{X\vdash\tau\cdot\varphi\mid \varphi\in\Phi\}\,
    \cup\,\{X\vdash\isdef\tau(f(i))\mid i\in\arity(\alpha)\}}
  {X\vdash\alpha(\tau\cdot f)}\;
  \begin{array}{l@{}l}
    (&\Phi\rimpl\alpha(f)\in\RelAx,\\
    &\tau\colon\Vars\to\TSigma(X))
  \end{array}
  \\[3mm]
  (\mathsf{Ax})\;
   \frac{\{X\vdash\alpha(\tau\cdot f)\mid\alpha(f)\in\Delta\}\,
   \cup\,\{X\vdash\isdef\tau(y)\mid y\in\Delta\}}{X\vdash\beta(\tau\cdot g)}~(\Delta\vdash\beta(g)\in\E)
\end{gather*}

\noindent
Recall that both the arities of operations
in~$\Sigma$ and the contexts of the $\kappa$-ary $\Sigma$-relations in
$\E$ are in~$\pres_{\kappa}$. We assume such a $\Delta\in\pres_\kappa$
to be specified as $\Delta=\Refl(Y,E)$ by a $\kappa$-presentable object
$(Y,E)\in\Str(\Pi)$ (cf.~\autoref{prop:fp-rels},
\autoref{prop:pres-object}, \autoref{lem:compact-generated},
\autoref{P:fpchar}); by writing $\phi\in\Delta$ for an edge~$\phi$, we
indicate that $\phi\in E$ (rather than just $\phi\in\Edge(\Delta)$).
The rules $(\EAr)$ and $(\IAr)$ apply to every $\sigma\in\Sigma$, and
rule $(\mathsf{Mor})$ applies to every $\sigma\in\Sigma$ and every
$\alpha\in\Rels$. The side condition $(+)$ of $(\IAr)$ is the
following: for some axiom $\Delta\vdash\gamma(g)$ of $\V$ there is
$\sigma(h)\in\sub(g)$, where
$h\colon\arity(\sigma)\to\TSigma(\Delta)$, such that
$\arity(\sigma)\models\beta(k)$ and
\[
  c = \arity(\beta)
  \xra{k}
  \arity(\sigma)
  \xra{h}
  \TSigma(\Delta)
  \xra{\tau}
  \TSigma(X).
\]
Rule $(\mathsf{Mor})$ captures the fact that operations~$\sigma$ are
interpreted as morphisms of type $[\arity(\sigma),A]\to A$, a
condition that relates to enrichment of the induced monad. Rule
$(\EAr)$ states that operations are defined when all the constraints
given by their arity are satisfied. Rules $(\mathsf{RelAx})$ and
$(\mathsf{Ax})$ allow application of the axioms of the Horn theory and
the variety, respectively, in both cases instantiated with a
substitution. A general substitution rule is not included but
admissible. Rule $(\IAr)$ captures that every axiom of the variety is
understood as implying that (under the constraints of the context) all
subterms occurring in it are defined, in the sense that the
constraints in the arities of the relevant operations hold.

\begin{rem}\label{R:relational-logic}
  Instantiating the above system of rules 
  to the theory of partial orders yields essentially the ungraded
  version of our previous deduction system for graded monads
  on~$\Pos$~\cite{FordEA21}, up to the above-mentioned coding of
  definedness judgements. At first glance, the instantiation to the theory 
  of metric spaces appears to yield a system that differs in several 
  respects from the existing system of quantitative algebra~\cite{MPP16}; 
  besides the mentioned absence of a general substitution rule, this 
  concerns most prominently the absence of a cut rule (included 
  in~\cite{MPP16}) in our system. These distinctions are only superficial: 
  as mentioned above, the more general substitution rule is
  admissible in our system, and it follows from completeness
  (\autoref{P:complete}) that the cut rule is admissible as well.
\end{rem}
\begin{lem}\label{L:admissible}
  The following rules are admissible:
  \[
    (\mathsf{Arity})
    \;
    \displaystyle
    \frac{X\vdash\isdef{\sigma(m)}}{X\vdash\alpha(m\cdot f)}
    ~~
    \begin{array}{@{}l@{}l}
      (& \arity(\sigma)\models\alpha(f), \\
      & m\colon |\arity(\sigma)| \to T_\Sigma(X))
    \end{array}
    \qquad
    (\mathsf{Subterm})
    \;
    \displaystyle
    \frac{X\vdash\alpha(f)}{X\vdash\isdef{u}}
    ~
    (u\in\sub(f))
  \]
 \end{lem}
\subparagraph*{Constructing free algebras}
We now show that
relational logic gives rise to a syntactic construction of free
algebras in the variety $\V$.

The set $\Term_{\V}(X)$ of \emph{derivably $\V$-defined terms
in $X$} consists of those terms $t\in\TSigma(X)$ such that 
$X\vdash\isdef t$ is derivable. We equip $\Term_{\V}(X)$ with the
relations 
\[
\Term_{\V}(X)\models\alpha(f)\Longleftrightarrow X\vdash\alpha(f)\text{ is derivable}
\qquad (\alpha\in\Rels, f\colon\arity(\alpha)\to\Term_{\V}(X))
\]
making it into a $\Rels$-structure. We write $\sim$ for the relation
on $\Term_{\V}(X)$ given by \emph{derivable equality}: that is, for
all $s,t\in\Term_{\V}(\Gamma)$ we put $s\sim t$ iff $X\vdash\varphi$
is derivable for all $\varphi\in\Eq(s, t)$, which is clearly an
equivalence relation. The $\sim$-equivalence class of
$t\in\Term_{\V}(X)$ is denoted by $[t]$. We pick a \emph{splitting}
$u\colon\Term_{\V}(X)/{\sim}\to\Term_{\V}(X)$ of the canonical
quotient map $q\colon\Term_{\V}(X)\to\Term_{\V}(X)/{\sim}$,
i.e.~$q\cdot u = \id$, so~$u$ picks representatives of $\sim$-equivalence
classes. Then $\Term_{\V}(X)/{\sim}$ carries the structure of a
$\BC$-object, with edges defined by
$\Term_{\V}(X)/{\sim}\models\alpha(f)$ iff
$\Term_{\V}(X)\models\alpha(u\cdot f)$.
(`Only if' means that $u$ is relation preserving.)

\begin{defn}
  The \emph{algebra $\F X$ of defined terms in $X$} is the
  $\Sigma$-algebra obtained by equipping $\Term_{\V}(X)/{\sim}$ with
  the operations
  $\sigma_{\F X}\colon[\arity(\sigma),
  \Term_{\V}(X)/{\sim}]\to\Term_{\V}(X)/{\sim}$ well-defined by
  $f\mapsto [\sigma(u\cdot f)]$, where
  $u\colon\F\Gamma\to\Term_{\V}(X)$ is the chosen splitting of
  $q\colon\Term_{\V}(X)\to\F X$.
\end{defn}

\begin{thm}\label{thm:free-alg}
For every $X\in\BC$, $\F X$ is a free algebra in $\V$.
\end{thm}

\begin{thm}[Soundness and Completeness]\label{P:complete}
$X\vdash\alpha(f)$ is derivable iff every $A\in\V$ satisfies 
$X\vdash\alpha(f).$
\end{thm}

\begin{rem}
  Consequently, our system instantiated to the theory of metric spaces
  and the system of quantitative algebra~\cite{MPP16}, which is also
  sound and complete, are deductively equivalent. Hence, our results
  thus far imply that every quantitative algebraic theory induces an
  $\omega_1$-accessible monad.  Indeed this remains true if one admits
  operations of countable arity, as in our theory of complete metric
  spaces (\autoref{expl:cmet}). Due to non-discrete contexts in
  axioms, monads induced by quantitative algebraic theories (such as
  $x=_{1/2} y\vdash x=_0y$) in general fail to be finitary. However, our
  results do imply that the induced monad is finitary if only
  discrete contexts are used; e.g.~this holds for the theories of
  left-invariant barycentric algebras and of quantitative
  semi-lattices, respectively~\cite{MPP16} (note for the latter that
  axiom (S4) can be omitted in \cite[Def.~9.1]{MPP16}). We conjecture
  that monads induced by \emph{continuous equation
    schemes}~\cite{MPP16} are also finitary.
\end{rem}
\section{Enriched Accessible Monads}\label{S:monad-theory}

\noindent We proceed to establish the monad-to-theory direction of our
correspondence; as already indicated, given our fixed $\lambda$-ary
Horn theory~$\Horn$, this works only for $\lambda$-accessible monads
and $\lambda$-ary theories, but not for accessibility degrees
$\kappa<\lambda$ as in the theory-to-monad direction. So let
$\bbT=(T, \eta, \mu)$ be an enriched $\lambda$-accessible monad on
$\BC$. We proceed to extract a $\lambda$-ary relational algebraic
theory from~$\bbT$. We first review the equivalence between monads and
\emph{Kleisli triples}~(see, e.g.,~Moggi~\cite{Moggi91}, and
originally Manes~\cite[Exercise
12]{Manes76}). 

\begin{defn}
  A \emph{Kleisli triple} in $\BC_0$ is a triple $(T, \eta, (-)^*)$
  consisting of a mapping $T\colon\BC_0\to\BC_0$ (of objects), a
  morphism $\eta_X\colon X\to TX$ for all $X\in\BC_0$, and an
  assignment of a morphism $f^*\colon TX\to TY$ to every morphism
  $f\colon X\to TY$. This data is subject to the laws below for all
  $X\in\BC_0$ and all morphisms $f\colon X\to TY$ and
  $g\colon Y\to TZ$:
  \begin{equation}\label{eq:kleisli-laws}
    \eta_X^* = \id_X,
    \qquad
    f^*\cdot\eta_X = f,
    \qquad\text{and}\qquad
    g^*\cdot f^* = (g^*\cdot f)^*.
  \end{equation}
\end{defn}

\begin{rem}\label{R:kleisli-mnd}
  The mapping which assigns to each monad $(T, \eta, \mu)$ the Kleisli
  triple $(T, \eta, (-)^*)$ with $(-)^*$ defined by
  \( f^* = TX\xrightarrow{Tf} TTY\xrightarrow{\mu_Y} TY \) for
  $f\in\BC_0(X, TY)$ yields a bijective correspondence between monads
  and Kleisli triples on $\BC_0$.
\end{rem}

\begin{notn}\label{N:operation}
  For each operation $\sigma$ in a signature $\Sigma$, we have a term
  $\sigma(u_{\arity(\sigma)})$, where $u_{\arity(\sigma)}$ is the
  inclusion $\arity(\sigma)\hookrightarrow \TSigma(\arity(\sigma))$.
  By abuse of notation, we also write $\sigma$ for $\sigma(u_{\arity(\sigma)})$.
\end{notn}

\begin{defn}\label{D:inducedtheory}
  The $\lambda$-ary signature $\Sigma_{\bbT}$ \emph{induced by~$\bbT$}
  is the disjoint union of the
  sets~$|T\Gamma|$~($\Gamma\in\pres_{\lambda}$), where elements of
  $|T\Gamma|$ have arity $\Gamma$.  The \emph{variety $\V_{\bbT}$
    induced by~$\bbT$} is $\V_{\bbT}=\Alg(\Sigma_{\bbT},\E_\bbT)$
  where~$\E_\bbT$ contains all axioms of the following shape, with
  $\Gamma\in\pres_{\lambda}$:
  \begin{enumerate}
  \item\label{D:inducedtheory-axiom-1}
    $\Gamma\vdash\alpha(\sigma_1,\dots, \sigma_{\arity(\alpha)})$ for
    all $\sigma_i\in T\Gamma$ such that
    $T\Gamma\models\alpha(\sigma_1,\dots, \sigma_{\arity(\alpha)})$
    
  \item\label{D:inducedtheory-axiom-2}
    $\Gamma\vdash f^*(\sigma)=\sigma(f)$ for all $\Delta\in\pres_{\lambda}$,
    morphisms $f\colon\Delta\to T\Gamma$,
    and $\sigma\in |T\Delta|$.

  \item\label{D:inducedtheory-axiom-3}
    $\Gamma \vdash \eta_\Gamma(x) = x$ for every $x \in \Gamma$.
  \end{enumerate}
\end{defn}
\noindent
Note that in the second item above, for every $x \in \Delta$ the
operation symbol $f(x) \in |T\Gamma|$ is considered as a term
according to \autoref{N:operation}. Hence $\sigma(f)$ is a term, too.

We now show that $\bbT$ is the free-algebra monad of its induced
variety $\V_{\bbT}$. For each $X\in\BC$, the $\BC$-object $TX$ carries
a canonical $\Sigma$-algebra structure with each operation
$\sigma_{TX}$ being defined by $\sigma_{TX}(f) := f^*(\sigma).$ We
call $TX$ the \emph{canonical algebra over $X$}.
\begin{lem}\label{L:TX-in-VT}
  Every canonical algebra lies in $\V_{\bbT}$.
\end{lem}
\begin{thm}\label{T:monad-theory}
  Each enriched $\lambda$-accessible monad
  $\bbT$ is the free-algebra monad of its induced variety
  $\V_{\bbT}$, with the free algebra on~$X$ given by the canonical algebra~$TX$.
\end{thm}
\begin{rem}
  We have thus shown that given a $\lambda$-ary Horn
  theory~$\Horn$, we we can translate
  $\lambda$-accessible monads on $\Str(\Horn)$ back into
  $\lambda$-ary theories, preserving the notion of model. For example, every
  $\omega_1$-accessible monad
  on~$\Met$ is induced by an
  $\omega_1$-ary theory, as illustrated in \autoref{expl:cmet}. The
  situation is more complicated for
  $\kappa$-ary monads where
  $\kappa<\lambda$. E.g.~we can generate a finitary monad
  on~$\Met$ from a single binary operation of type $\{(x,y)\in A^2\mid
  d(x,y)<1/2\}\to
  A$. This monad is not induced by any theory with operations of
  internally finitely presentable (i.e.~discrete) arity, in particular
  neither by an
  $\omega$-ary theory in our framework nor by a quantitative algebraic
  theory~\cite{MPP16}.
\end{rem}
\section{Conclusions}

\noindent We have introduced the framework of \emph{relational logic}
for reasoning about algebraic structure on categories of (finitary)
relational structures axiomatized by possibly infinitary Horn
theories, such as partial orders or metric spaces. We have proved
soundness and completeness of a generic algebraic deduction system,
and we have shown that
$\lambda$-ary relational algebraic theories are in correspondence with
$\lambda$-accessible enriched monads when the underlying Horn theory 
is also $\lambda$-ary (where
`$\lambda$-ary' refers to the arity of operations for relational algebraic
theories, and to the number of premisses in axioms for Horn
theories). Our results allow for a straightforward specification also
of infinitary constructions such as metric completion.

The theory-to-monad direction of the above-mentioned correspondence
remains true for $\kappa$-ary relational algebraic theories and
$\kappa$-accessible monads on categories of models of $\lambda$-ary
Horn theories for $\kappa < \lambda$, e.g.~when looking at monads and
theories on metric spaces. One open end that we leave for future
research is to obtain a more complete coverage of this case, which
will require a substantial generalization of both the way arities of
operations are defined (these can no longer be taken to be objects of
the base category) and in the way the axioms of the theory are
organized, likely using more topologically-minded approaches.

\takeout{                     
\section{Setting}

\texttt{To be determined precisely, very general at the moment, maybe
  go down to monoidal or even Cartesian closed later. $\BC$ should
  probably be a category of relational structures.}

Let~$\BV$ be a closed symmetric~\lsnote{At
  least~\cite{LucyshynWright15} assumes this} monoidal category, let
$\BC$ be a cotensored $\BV$-enriched category, with cotensors/powers
denoted $V\power C$ (recall that these are characterized as being
equipped with a natural isomorphism
$\BC(D,V\power C)\cong \BV(V,\BC(D,C))$), and let~$\CJ$ be a set of
$\BV$-objects called \emph{arities}. A \emph{signature} is a
set~$\Sigma$ of \emph{operation symbols}~$\sigma$ with assigned
arities $\arity(\sigma)\in\CJ$. A \emph{$\Sigma$-algebra}~$A$ consists
of a $\BC$-object, also denoted~$A$ by abuse of notation, and for each
$\sigma\in\Sigma$ a $\BC$-morphism
\begin{equation*}
  \sigma^A\colon \arity(\sigma)\power A\to A.
\end{equation*}
A \emph{homomorphism} $A\to B$ of $\Sigma$-algebras $A,B$ is a
$\BC$-morphism $h\colon A\to B$ such that for each $\sigma\in\Sigma$,
the diagram
\begin{equation*}
  \begin{tikzcd}
    \arity(\sigma)\power A \ar{r}{\sigma^A}
    \ar{d}[left]{\arity(\sigma)\power h}& A \arrow{d}[right]{h}\\
    \arity(\sigma)\power B \ar{r}{\sigma^B} & B
  \end{tikzcd}
\end{equation*}
commutes.

A \emph{$\BV$-enriched monad} on~$\BC$ consists of an enriched functor
on $\BC$, equipped with enriched natural transformations satisfying
the usual equational axioms cast as commutative diagrams
in~$\BV$.\lsnote{Eventually make this more explicit}
}

\newpage
\bibliography{monadsrefs}

\clearpage
\appendix

\section{Details for \autoref{sec:relationalstructures}}\label{sec:app-relationalstructures}

\subparagraph*{Details for~\autoref{E:HornCategories}}

\emph{$L$-valued relations:} Every $\Horn_L$-model~$(X,E)$ induces an
$L$-valued relation $P_E$ on $X$ by
\[
  P_E(x, y)=\bigwedge\{p\in L_0\mid \alpha_p(x,y)\in E\}.
\]
Conversely, every $L$-valued relation $P$ on a set~$X$ induces a
$\Horn_L$-model $(X,E_P)$ given by
\[
  E_P=\{\alpha_p(x,y)\mid p\in L_0,  p\ge P(x, y) \}.
\]
The given conditions on~$L_0$ ensure precisely that these
constructions are mutually inverse: By meet density of~$L_0$, it is
straightforward to see that $P_{E_P}=P$. Moreover, given a
$\Horn_L$-model $(X,E)$, we have $E_{P_E}=E$. Here, `$\supseteq$' is
trivial; we prove `$\subseteq$'. So suppose that
\begin{equation*}
  q\ge \bigwedge\{p\in L_0\mid \alpha_p(x,y)\in E\};
\end{equation*}
we have to show $\alpha_q(x,y)\in E$. We
distinguish the following cases:
\begin{enumerate}
\item If $r:=\bigwedge\{p\in L_0\mid \alpha_p(x,y)\in E\}\in L_0$,
  then by \textbf{(Arch)}, $\alpha_r(x,y)\in E$, and hence
  $\alpha_q(x,y)\in E$ by \textbf{(Up)}.
\item Otherwise, the assumptions on~$L_0$ imply that there
  is~$p\in L_0$ such that $\alpha_p(x,y)\in E$ and $q\ge p$; then
  $\alpha_q(x,y)\in E$ by \textbf{(Up)}.
\end{enumerate}

\emph{Metric spaces:} As indicated in the main text, this is
essentially a special case of the one for $L$-valued relations:
$L_0:=[0,1]\cap\Q\subseteq [0,1]=:L$ satisfy the relevant assumptions,
so we have mutually inverse constructions as indicated above and in
the main text. Under this correspondence, the axioms of metric spaces
clearly correspond precisely to the axioms \textbf{(Refl)},
\textbf{(Equal)}, \textbf{(Sym)}, and \textbf{(Triang)}
of~$\Horn_\Met$.

\begin{proof}[{Proof of \autoref{prop:pres-object}}]

\ref{item:refl-pres}$\implies$\ref{item:pres}: Generally, left adjoint
functors (in this case,~$\Refl$) preserve $\lambda$-presentable
objects if their right adjoint (in this case, the inclusion
$\Str(\Rels,\RelAx)\into\Str(\Rels)$) preserves $\lambda$-directed
colimits (see e.g.~\cite[Lem.~2.4(1)]{AdamekEA19}).

\ref{item:pres}$\implies$\ref{item:generated}: The bound on the
cardinality is as in the more general case (see the proof of
\autoref{P:fpchar}). We can assume that $r_X=\id_X$. Write $X$ as the
$\lambda$-directed colimit in~$\Str(\Rels)$ of the objects $(|X|,E)$
such that $E\subseteq \Edge(X)$ and $\card E<\lambda$. This colimit is
preserved by the left adjoint~$\Refl$. Now apply
$\lambda$-presentabiliy of~$X$ to see that
$r_X=\id_X\colon X\to \Refl X$ factorizes through $Ri$ for one of the
inclusions $i\colon (|X|, E) \to X$; that is, $Ri$ is a split
epimorphism. The claim follows once we show that~$Ri$ is also a
monomorphism. To this end, consider the naturality square
\begin{equation*}
  \begin{tikzcd}[column sep=large]
    (|X|,E)  \arrow[r,"r_{(|X|,E)}"] \arrow[d,"i"] & \Refl(|X|,E)
    \arrow[d,"\Refl i"]\\
    X \arrow[r,"r_X" below] & \Refl X
  \end{tikzcd}
\end{equation*}
In this diagram, the underlying maps of~$i$ and $r_X$ are bijective,
so the underlying map of $r_{(|X|,E)}$ is injective, and hence also
bijective since~$r$ is componentwise surjective. It follows that the
underlying map of~$\Refl i$ is bijective, hence~$\Refl i$ is monic.

\ref{item:generated}$\implies$\ref{item:refl-pres}: Take
$(Y,E)=(|X|,E)$.
\end{proof}

\begin{proof}[Proof of \autoref{P:closedmonoidal}]
We show that~$u_Y$ is indeed a universal arrow; that is, we show
that~$u_Y$ is actually a morphism of the claimed type, and that~$u_Y$
has the following universal property: Given a further
$\Str(\Rels)$-morphism $g\colon Y\to[X,Z]$, there exists a unique
morphism $g^\sharp\colon Y\otimes X\to Z$ such that $[X,g^\sharp]=g$:
\begin{equation*}
  \begin{tikzcd}
    Y \arrow[r,"u_Y"]
     \ar{dr}[swap]{\forall g}
    & { [X,Y\otimes X]}
    \arrow[d,"{[X,g^\sharp]}"]
    &
    Y \otimes X
    \ar[dashed]{d}{\exists! g^\sharp}
    \\
    & {[X,Z]} & Z
  \end{tikzcd}
\end{equation*}
We prove the latter property first. We put
\begin{equation*}
  g^\sharp(y,x)=g(y)(x).
\end{equation*}
This is clearly the only map making the diagram commute, so we are
only left to show that it is indeed a morphism. So let
$e\in\Edge(Y\otimes X)$; we have to show $g^\sharp\cdot
e\in\Edge(Z)$. We distinguish cases on~$e$ according to the definition
of $Y\otimes X$:
\begin{itemize}
\item $\pi_1\cdot e$ is constant, say with constant value~$y$, and
  $\pi_2\cdot e\in\Edge(X)$. Then
  $g^\sharp\cdot e=g(y)\cdot\pi_2\cdot e\in\Edge(Z)$ because $g(y)$ is a
  morphism\lsnote{This part may fail if edges can have arity~$0$},
  and
\item $\pi_2\cdot e$ is constant, say with constant value~$x$, and
  $\pi_1\cdot e\in\Edge(Y)$. Then
  $g^\sharp\cdot e=\pi_x\cdot g\cdot e\in\Edge(Z)$ because~$g$ is a
  morphism.
\end{itemize}
We next show that $u_Y$ has the claimed type; that is, we show that
for $y\in Y$, $u_Y(y)\colon X\to Y\otimes X$ is a morphism. So let
$e\in\Edge(X)$. Then $u_Y(y)\cdot e\in\Edge(Y\otimes X)$ as required,
since $\pi_1\cdot u_Y(y)\cdot e$ is constant and
$\pi_2\cdot u_Y(y)\cdot e=e\in\Edge(X)$. Finally, we show that~$u_Y$
is a morphism. So let $e\in\Edge(Y)$, and let $x\in X$; we have to
show that $\pi_x\cdot u_Y\cdot e\in\Edge(Y\otimes X)$. This holds
because $\pi_2\cdot \pi_x\cdot u_Y\cdot e$ is constant and
$\pi_1\cdot \pi_x\cdot u_Y\cdot e=e\in\Edge(Y)$.
\end{proof}

\takeout{\subsection{Proof of \autoref{P:closedmonoidal}}
We show that
\[
  \ev_Y\colon [X, Y]\otimes X\to Y, \quad (f, x)\mapsto f(x)
\]
is indeed a couniversal arrow, i.e.~is a morphism in $\Str(\Rels)$ and
enjoys the following universal property: For every morphism
$g\colon Z\otimes X\to Y$ in $\Str(\Rels)$, there exists a unique
morphism $\widetilde{g}\colon Z\to [X, Y]$ such that
$g=\ev_Y\cdot (\widetilde{g}\otimes X)$, all as in the diagram below:
\begin{equation}\label{eq:curry}
  \begin{tikzcd}
    {[X, Y]}\otimes X
    \ar{r}{\ev_Y}
    &
    Y
    \\
    Z\otimes X
    \ar{ru}[swap]{g}
    \ar[dashed]{u}{\tilde{g}\otimes X} 
  \end{tikzcd}
  \qquad
  \begin{tikzcd}
    {[X,Y]}
    \\
    Z \ar[dashed]{u}{\tilde{g}}
  \end{tikzcd}
\end{equation}
We prove the latter property first, defining $\widetilde{g}$ by
\[
  \widetilde{g}(z)= \lambda x.\, g(z, x).
\]
Clearly, $\widetilde g$ is the unique map making the desired triangle
in~\eqref{eq:curry} commute. Hence, we only need to prove that it is a
morphism in $\Str(\Rels)$.

To this end, let $e\in\Edge(Z)$; 
we have to show that $\widetilde{g}\cdot e\in\Edge([X, Y])$, i.e.~that
$\pi_x\cdot\widetilde{g}\cdot e\in\Edge(Y)$ for all $x\in X$.
Let $x\in X$, and define $k_x\colon Z\to Z\times X$ by
$k_x(z)=(z,x)$. Observe that $k_x\cdot e\in\Edge(Z\otimes X)$, since
$\pi_1\cdot k_x\cdot e = e\in\Edge(Z)$ and $\pi_2\cdot k_x\cdot e$ is
constant. Thus, $g\cdot k_x\cdot e\in\Edge(Y)$, since~$g$ is a
morphism. However, $g\cdot k_x = \pi_x\cdot\widetilde g$, so we are
done. 

It remains to show that $\ev_Y$ is a morphism. So let~$e$ be an edge
in $[X,Y]\otimes X$; we have to show that $\ev_Y\cdot
e\in\Edge(Y)$. We distinguish cases on~$e$ according to the definition
of~$\otimes$:
\begin{itemize}
\item $\pi_1\cdot e=(\alpha,f)$ is constant, say with~$f$ having
  constant value $h\in[X,Y]$, and $\pi_2\cdot e\in\Edge(X)$. Then
  $\ev_Y\cdot e = h\cdot (\pi_2\cdot e)\in\Edge(Y)$ because~$h$ is a
  morphism.
\item $\pi_2\cdot e=(\alpha,f)$ is constant, say with~$f$ having
  constant value $x\in X$,and $\pi_1\cdot e\in\Edge([X,Y])$. Then
  $\ev_Y\cdot e=\pi_x\cdot(\pi_1\cdot e)\in\Edge(Y)$ by definition of
  $\Edge([X,Y])$.
\end{itemize}
\qed
}

\noindent We have claimed that the internal hom restricts to
$\Str(\Horn)$, that is:
\begin{lem}\label{L:internalhom}
$[X, Y]\in\Str(\Rels, \RelAx)$ for all $Y\in\Str(\Rels, \RelAx)$.
\end{lem}

\begin{proof}[Proof of~\autoref{L:internalhom}]
Let $\Phi\rimpl\beta(g)$ be an axiom and let 
$e\colon\Vars\to\Str(\Rels)(X, Y)$ be a map such that 
$\alpha(e\cdot f)\in\Edge([X, Y])$ for all $\alpha(f)\in\Phi.$
We proceed to show that $\beta(e\cdot g)\in\Edge([X, Y])$.
That is, we verify that $\beta(e\cdot g)_x)\in\Edge(Y)$ 
for all $x\in X$. 

To this end, it suffices to show that 
$\beta(\xi_x\cdot g)\in\Edge(Y)$ for all $x\in X$, where 
$\xi_x\colon\Vars\to Y$ is the assignment 
$v\mapsto e(v)(x).$ Indeed, we have
\[
\xi_x\cdot g(i)=\xi_x(g(i))=e(g(i))(x)= (e\cdot g)_x(i)
\]
for all $i\in\arity(\beta)$. Thus, $\xi_x\cdot g=(e\cdot g)_x$.
To conclude the proof, note that for each $\alpha(f)\in\Phi$
we have $\alpha((e\cdot f)_x)\in\Edge(Y)$ for all $x\in X$. 
Hence also $\alpha(\xi_x\cdot f)\in\Edge(Y)$ for all $x\in X$. 
Since $Y$ is an $\RelAx$-model, it follows that 
$\beta(\xi_x\cdot g)\in\Edge(Y)$ for all $x\in X$, as desired. 
\end{proof}

\begin{proof}[Proof of \autoref{c:tensor-refl}]
  From \autoref{P:closed} below, using
  \autoref{L:internalhom}.
\end{proof}
\begin{propn}\label{P:closed}
  Let $\iota\colon \CatC'\hookrightarrow\CatC$ be a full reflective
  subcategory, with reflector~$R\colon \CatC \to \CatC'$, of a closed
  monoidal category $(\CatC, \otimes, I)$ such that $[C, D]$ lies in
  $\CatC'$ for every $C, D \in \CatC'$. Then $\CatC'$ is also closed
  monoidal, with monoidal product $R(X\otimes Y)$.
\end{propn}
\begin{proof}
  We have the following chain of bijections, natural in $Y$ and $Z$
  (and omitting $\iota$ in front of $X$ and $Y$):
  \[
    \CatC'(R(Y \otimes X),  Z)
    \cong 
    \CatC (Y \otimes X, \iota Z)
    \cong
    \CatC( Y, [X,\iota Z])
    \cong
    \CatC'(Y, [X, Z]),
  \]
  where the last step uses that $[X, \iota Z]$ lies in $\CatC'$ and
  $\CatC'$ is full. 
\end{proof}

\begin{proof}[Proof of~\autoref{prop:internally-pres}]
  One extracts from~\cite[Proof of Proposition~2.23]{AR94} that if $F$
  is a right adjoint functor between locally $\lambda$-presentable
  categories whose left adjoint preserves externally
  $\lambda$-presentable objects, then~$F$ is $\lambda$-accessible. We
  apply this to $F=[X,-]$ for~$X$ externally
  $\lambda$-presentable. Preservation of externally
  $\lambda$-presentable objects by the left adjoint $(-)\otimes X$
  of~$F$ means precisely that the class of externally
  $\lambda$-presentable objects is closed under~$\otimes_\Horn$ as
  claimed in the main body of the paper. By
  \autoref{prop:pres-object}, this follows immediately from the
  corresponding statement for $\Horn=(\Rels,\emptyset)$. But the
  latter is clear by the description of externally
  $\lambda$-presentable objects in~$\Str(\Rels)$ in terms of simple
  cardinality constraints.
\end{proof}

\section{Details for~\autoref{S:presentableobjects}}\label{sec:app-compact}

\subparagraph*{Details for \autoref{expl:internal-fp}}

\ref{item:non-distrib}: \emph{Example of a complete lattice~$L$ where
  binary joins fail o distribute over directed infima:} Take $L$ to
consist of two copies of the reals with their usual orderings, with
elements denoted $l(x)$ and $r(x)$, respectively, for
$x\in\mathbb{R}$, and additionally a top element~$\top$, a bottom
element~$\bot$, and an element~$b$ that satisfies $l(0)\le b$,
$r(0)\le b$, and no other inequalities other than the ones entailed
by~$\top$ and $\bot$ being a top and a bottom element,
respectively. Then $\bigwedge_n (l(1/n)\vee r(1/n)) = \top$ but
$(\bigwedge_n l(1/n))\vee(\bigwedge_n(r(1/n))=l(0)\vee r(0)=b$.

\ref{item:non-distrib-fp}: In the same~$L$, there is also an element
as prescribed in \ref{item:non-distrib-fp}, namely $l=l(1)$.

\begin{proof}[Proof of \autoref{lem:compact-generated}]
  Given $(X,E)$ $\kappa$-compact, apply $\kappa$-compactness to the
  cover~$E$ of~$(X,E)$.
\end{proof}

\noindent We next verify the characterization of weakly
$\kappa$-presentable objects in the locally $\lambda$-presentable
category $\BC$, where $\kappa<\lambda$ is an infinite cardinal, as
precisely the $\kappa$-compact objects of cardinality $<\lambda$
(\autoref{P:fpchar})

The main essence of the characterization is the captured inn the
following lemmas:

\begin{lem}\label{lem:compact}
  Let $X\in\Str(\Rels,\RelAx)$.
  \begin{enumerate}
  \item $X$ is $\kappa$-compact iff for every morphism of the form
    $f\colon X\to\Refl(Y,E)$ for $(Y,E)\in\Str(\Rels)$, there exist
    $E'\subseteq_\kappa E$ and $f'\colon X\to\Refl(Y,E')$ such that
    $f=\Refl j\cdot f'$, where $j\colon (Y,E')\to(Y,E)$ is the
    $\Str(\Rels)$-morphism carried by~$\id_Y$.
  \item If~$\card|X|<\kappa$, then $X$ is $\kappa$-compact iff for
    every morphism of the form $f\colon X\to\Refl(Y,E)$ for
    $(Y,E)\in\Str(\Rels)$, there exist $Y'\subseteq_\kappa Y$,
    $E'\subseteq_\kappa E$ and $f'\colon X\to\Refl(Y',E')$ such that
    $f=\Refl j\cdot f'$, where $j\colon (Y',E')\to(Y,E)$ is the
    $\Str(\Rels)$-morphism carried by the inclusion $Y'\into Y$.
\end{enumerate}
\end{lem}
\begin{proof}
  \begin{enumerate}
  \item `If' is trivial; we prove `only if'. For ease of notation, we
    assume w.l.o.g.~that $|X|$ and~$Y$ are disjoint; then their union
    $|X|\cup Y$ serves as a coproduct $|X|+Y$. Fix a
    map~$s\colon |\Refl(Y,E)| \to Y$ splitting the reflection
    $r_{(Y,E)}\colon (Y,E)\to \Refl(Y,E)$; that is,
    $r_{(Y,E)} \cdot s = \id_Y$ ($s$ is not in general a morphism).%
    We define a set~$E_f$ of edges over $|X| + Y$ by
    \begin{equation*}\textstyle
      E_f=\bigcup_{x\in |X|}\Eq(x,s(f(x))).
    \end{equation*}
    We then have a morphism
    \begin{equation*}
      h = \big((Y,E)
      \xra{\inr}
      (|X| + Y, E_f \cup E)
      \xra{r_{(|X|+Y, E_f \cup E)}}
      \Refl(|X| + Y, E_f \cup E)\big).
    \end{equation*}
    Note that $r_{(|X| + Y,E_f \cup E)}$ merges every $x \in |X|$ with
    $e(f(x)) \in Y$, and indeed it is not difficult to see that
    $h^\sharp\colon\Refl(Y,E)\to\Refl(|X| + Y,E_f \cup E)$ is an
    isomorphism.  In addition, the map
    $r_{(|X| + Y,E_f \cup E)}\cdot \inl\colon X\to\Refl(|X| + Y,E_f
    \cup E)$ carries a morphism, where~$\inl$ is the left-hand
    coproduct injection\lsnote{This should be introduced globally}
    $|X| \to |X| + Y$; indeed,
    $r_{(|X| + Y,E_f \cup E)}\cdot \inl =h^\#\cdot f$. By assumption,
    we thus have $E'\subseteq_\kappa E_f \cup E$ and a morphism
    $l\colon X\to\Refl(|X|\cup Y,E')$ such that
    $\Refl j\cdot l=r_{(|X|+ Y,E_f\cup E)}\cdot \inl$ where
    $j\colon (|X| + Y,E')\to(|X| + Y,E_f \cup E)$ is the morphism
    carried by~$\id_{|X| + Y}$.

    Next, we define the map
    \begin{equation*}
      k = [s \cdot f, \id_Y]\colon |X| + Y \to Y.
    \end{equation*}
    Then we have a morphism
    \begin{equation*}
      p=r_{(Y,E'\cap E)}\cdot k\colon (|X| + Y,E')\to \Refl(Y,E'\cap E).
    \end{equation*}
    To see this, let $(\alpha,g)\in E'\subseteq E_f\cup E$. The case
    where $(\alpha,g)\in E$ is trivial; the other case is that
    $(\alpha,g)\in\Eq(x,s(f(x)))$ for some $x\in X$. Such an edge is
    preserved by~$p$ because $p(x)= r_{(Y,E'\cap E)} \cdot s \cdot f(x) =
    r_{(Y,E'\cap E)} (s(f(x))) = p(s(f(x))$.
    
    We thus have a diagram
    \begin{equation*}
      \begin{tikzcd}[column sep = large]
        &
        \Refl(|X| + Y,E')  \arrow[d,"\Refl j"]
        \arrow[dr,"p^\sharp" above right]
        \\
        X\arrow[ur,"l"]
        \arrow[r,"r\cdot \inl"]
        \arrow[dr,"f" below left]
        &
        \Refl(|X| + Y,E_f\cup E)
        \arrow[d,"(h^\sharp)^{-1}"]
        &
        \Refl(Y,E'\cap E)\arrow[dl,"\Refl m"]
        \\
        &
        \Refl(Y,E)
      \end{tikzcd}
    \end{equation*}
    where $m\colon (Y,E'\cap E)\to(Y,E)$ is the morphism carried
    by~$\id_Y$. We are done once we show that the diagram
    commutes. The upper left triangle commutes by construction of~$l$,
    and commutation of the lower left triangle has been noted above
    (in the equivalent form $h^\sharp\cdot r\cdot \inl=f$). Commutation
    of the right-hand square is equivalent to
    $h^\sharp\cdot\Refl m\cdot p^\sharp = \Refl j$. This holds 
    precisely when it holds when precomposed with the reflection
    $r_{(|X|+Y, E')}$; we compute as follows (in $\Set$):
    \begin{align*}
      h^\sharp\cdot\Refl m\cdot p^\sharp\cdot r_{(|X|+Y, E')} 
      &= h^\sharp \cdot \Refl m\cdot p
      & \text{univ.~prop.~of $r_{(|X|+Y, E')}$} \\
      &= h^\sharp \cdot \Refl m \cdot r_{(Y, E'\cap E)} \cdot k
      & \text{def.~of $p$}\\
      &= h^\sharp \cdot r_{(Y,E)} \cdot m \cdot k &
      \text{def.~of $\Refl$ on morphisms} \\
      &= h^\sharp \cdot r_{(Y,E)} \cdot k
      & \text{$m$ carried by identity} \\
      &= h \cdot k
      &\text{univ.~prop.~of $r_{(Y,E)}$}\\
      &= h \cdot [s \cdot f, \id_Y]
      &\text{def.~of $k$}\\
      &= r_{(|X|+Y,E_f + E)} \cdot \inr \cdot [s \cdot f, \id_Y]
      &\text{def.~of $h$}\\
      &\overset{(*)}{=} r_{(|X|+Y,E_f + E)} \cdot [\inl, \inr]\\
      &= r_{(|X|+Y,E_f + E)} \cdot j
      &\text{$j$ carried by identity}\\
      &= Rj \cdot r_{(|X|+Y, E')}
      &\text{def.~of $\Refl$ on morphisms}.
    \end{align*}
    For the step marked with $(*)$ we consider coproduct components
    separately: the right-hand component is clear, and for the
    left-hand one use that $r_{(|X| + Y, E_f \cup E)}$ merges $\inr \cdot
    s \cdot f$ and $\inl$ since $E_f \supseteq \Eq(x, s(f(x)))$.
    \takeout{
    We generally
    denote equivalence classes induced by a set~$E$ of edges
    under~$\RelAx$\lsnote{This should be introduced globally in the
      Horn section} by~$[-]_E$, and calculate explicitly,
    distinguishing cases on $|X|\cup Y$:%
    \smnote{This can be done without equivalence classes by
      precomposing the desired equation by $r_{(|X| + Y, E')}$ and
      then consider coproduct components, separately.}  
    For $x\in |X|$, we have
    \begin{multline*}
      h^\sharp\cdot\Refl m\cdot p^\sharp([x]_{E'})=
      h^\sharp\cdot\Refl m([e(f(x))]_{E'\cap E}) =
      h^\sharp([e(f(x))]_E) \\ =[e(f(x))]_{E_f\cup E} =
      [x]_{E_f\cup E} = \Refl j([x]_{E'})
    \end{multline*}
    using in the second-to-last step that $E_f\supseteq\Eq(x,e(f(x)))$. For
    $y\in Y$, we have
    \begin{equation*}
      h^\sharp\cdot\Refl m\cdot p^\sharp([y]_{E'})=
      h^\sharp\cdot\Refl m([y]_{E'\cap E})=
      h^\sharp([y]_{E}) =
      [y]_{E_f\cup E}=\Refl j([y]_{E'}).
    \end{equation*}}
  \item Immediate from the first claim and \autoref{lem:edge-finite}.\qedhere
  \end{enumerate}
\end{proof}

\begin{lem}\label{lem:edge-finite}
  Let $f\colon X\to\Refl(Y,E)$, and suppose that
  $\card|X|,\card E<\kappa$. Then there exist $Y'\subseteq_\kappa Y$
  such that all edges in~$E$ are edges over~$Y'$, and
  $f'\colon X\to\Refl(Y',E)$ such that $f=\Refl i\cdot f'$ where
  $i\colon(Y',E)\into(Y,E)$ is the morphism carried by the inclusion
  $Y'\into Y$.
\end{lem}
\begin{proof}
  Fix a splitting $s\colon |\Refl(Y,E)|\to Y$ of the reflection
  $r_{(Y,E)}$ ($s$ is not in general a morphism). Write $Y_E$ for the
  set of elements of~$Y$ that appear in edges in~$E$, and put
  \begin{equation*}
    Y'=e\cdot f[X]\cup Y_E;
  \end{equation*}
  if~$Y'$ is empty, then add an arbitrary element of~$Y$ to~$Y'$
  (if~$Y$ is also empty, then there is nothing to show).  Then
  $\card Y' < \kappa$, and since $Y'\neq\emptyset$, the inclusion map
  $i\colon Y'\into Y$ has a left inverse $q\colon Y \to Y'$, i.e.\
  $q \cdot i = \id_{Y'}$. 
  Since $s(f(x)) \in Y'$ for all $x \in |X|$, we have
  $i \cdot q \cdot s \cdot f = s \cdot f$. Moreover, $q$ is clearly a
  morphism $q\colon(Y,E)\to(Y',E)$. The arising morphism
  $f'=\Refl q\cdot f$ provides the claimed factorization of~$f$:
  \begin{align*}
    \Refl i \cdot \Refl q \cdot f
    &= \Refl(i \cdot q) \cdot f
    &\text{functoriality of $\Refl$}\\
    &= \Refl(i \cdot q) \cdot r_{(Y,E)} \cdot s \cdot f
    &\text{since $r_{(Y,E)} \cdot s = \id$}\\
    &= r_{(Y,E)} \cdot i \cdot q \cdot s \cdot f
    &\text{def.~of $\Refl$ on morphisms}\\
    &= r_{(Y,E)} \cdot s \cdot f
    &\text{since $e\cdot f = i\cdot q\cdot s \cdot f$}\\
    &= f
    &\text{since $r_{(Y,E)} \cdot s = \id$}. \tag*{\qedhere}
  \end{align*}
  \takeout{
  Indeed, writing~$[-]_E$ for equivalence classes induced
  by~$E$ under~$\RelAx$ both in $\Refl(Y,E)$ and in
  $\Refl(Y',E)$\lsnote{Again, introduce some global notation for
    this}, we have
  \begin{equation*}
    \Refl i\cdot \Refl q\cdot f(x) = \Refl(i\cdot q)([e\cdot f(x)]_E)=
    [i\cdot q\cdot e\cdot f(x)]_E = [e\cdot f(x)]_E=f(x)
  \end{equation*}
  for $x\in X$, using in the second-to-last step that
  $e\cdot f(x)\in Y'$.}
\end{proof}

\begin{proof}[Proof of~\autoref{P:fpchar}]
  \ref{P:fpchar-itm1}$\Rightarrow$\ref{P:fpchar-itm2}:
  We first prove that $\card|X| < \kappa$. Write $X$ as the
  $\kappa$-directed colimit of all its subobjects of cardinality smaller than
  $\kappa$. This is possible by our assumptions that arities of
  relation symbols in $\Rels$ are finite. 
  Since $X$ is weakly $\kappa$-presentable, we know that $\id_X$
  factorizes through one of the colimit injections, say $i\colon X'
  \to X$ for $X'$ with $\card|X'| < \kappa$. That is, for some
  morphism $h\colon X \to X'$ we have $i \cdot h = \id_X$, whence $X
  \cong X'$ has cardinality smaller than $\kappa$. 

  Second, we prove that $X$ is $\kappa$-compact.  Given a
  cover~$(Y,E)$ of~$X$, with inclusion $i\colon|X|\into Y$, write
  $(Y,E)$ as the $\kappa$-directed colimit in $\Str(\Rels)$ of all
  objects $(Y,E')$ such that $E'\subseteq_{\kappa}
  E$. 
  Since the reflector~$\Refl$ is a left adjoint, it preserves all
  colimits, so $\Refl (Y,E)$ is the $\kappa$-directed colimit of the
  $\Refl (Y,E')$. By weak $\kappa$-presentability of~$X$, the morphism
  $r_{(Y,E)}\cdot i\colon X\to\Refl(Y,E)$ factors through one of the
  colimit injections, as required.

  \ref{P:fpchar-itm2}$\Rightarrow$\ref{P:fpchar-itm1}: Let
  $(D_i\xrightarrow{c_i} D)_{i\in I}$ be a directed colimit of objects
  $D_i=(Y_i,E_i)$ in $\Str(\Rels, \RelAx)$, and let $f\colon X\to D$
  be a morphism. 
  Let $(D_i\xrightarrow{\bar c_i}(Y,E))_{i\in I}$ be the colimit of
  the~$D_i$ in $\Str(\Rels)$; then $D=\Refl(Y,E)$ and
  $c_i=r_{(Y,E)}\cdot \bar c_i=\Refl(\bar c_i)$ (assuming
  w.l.o.g.~that $\Refl(Y_i,E_i)=(Y_i,E_i)$). 
  By \autoref{lem:compact}, there exist $Y'\subseteq_\kappa Y$,
  $E'\subseteq_\kappa E$, and $f'\colon X\to\Refl(Y,E')$ such that
  $f=\Refl j\cdot f'$ where $j\colon (Y',E')\to(Y,E)$ is the morphism
  carried by the inclusion $Y'\into Y$. Now $(Y',E')$ is
  $\kappa$-presentable in $\Str(\Rels)$; so there are~$i$ and
  $j'\colon (Y',E')\to(Y_i,E_i)$ such that $j=\bar c_i\cdot j'$. This
  yields the desired factorization of~$f$: For $f''=\Refl j'\cdot f'$,
  we have
  \begin{equation*}
    c_i\cdot f''=\Refl\bar c_i\cdot \Refl j'\cdot f'
    = \Refl j\cdot f' = f.\tag*{\qedhere}
  \end{equation*}
\end{proof}

\section{Details for~\autoref{sec:algebras}}\label{sec:app-algebras}

We will now verify that the object assignment
$I\colon\AlgSigma\to\Alg H_{\Sigma}$ forms a concrete isomorphism
over~$\BC$. An element of $H_{\Sigma}X$ is a morphism
$f\colon \arity(\sigma)\to X$ in $\BC$ for some specified
$\sigma\in\Sigma$; we denote this by the pair $(\sigma, f)$.

Indeed, it is easy to
verify that the assignment $I\colon(A, \sigma_A)\mapsto (A, \alpha)$
which maps each $\Sigma$-algebra $A$ to the $H_{\Sigma}$-algebra
$\alpha\colon H_{\Sigma}A\to A$ defined by
\[
\alpha(f\colon\arity(\sigma)\to A):=\sigma_A(f)
\]
is an isomorphism, and $I$ preserves the 
$\BC$-object underlying each $\Sigma$-algebra
and the underlying morphism of each homomorphism
in $\AlgSigma$. That is, for the forgetful functors
\[
  U\colon\AlgSigma\to\BC_0
  \qquad\text{and}\qquad
  \overline{U}\colon\Alg H_{\Sigma}\to\BC_0,
\]
we have $U = I\cdot\overline{U}$.
\takeout{   
 \[
  \begin{tikzcd}[column sep = 10]
    \AlgSigma \arrow[rd, "U"'] \arrow[rr, "I"] &     
    & \Alg H_{\Sigma} \arrow[ld, "\overline{U}"] \\
    & \BC_0 &                                           
  \end{tikzcd}
\]
}
\begin{propn}\label{P:algebra-iso}
$\AlgSigma$ and $\Alg H_{\Sigma}$ are
isomorphic as concrete categories over 
$\BC$.
\end{propn}
\begin{proof}
Every $H_{\Sigma}$-algebra $\alpha\colon H_{\Sigma}A\to A$ 
induces a $\Sigma$-algebra $A$, on the same carrier~$A$, with 
the operations $\sigma_A$~$(\sigma\in\Sigma)$ defined by
\[
  \sigma_{A}(f):= \alpha(\sigma, f)
  \qquad\qquad
  \text{for every $f\colon \arity(\sigma) \to A$}.
\]
Conversely, each $\Sigma$-algebra $A$ induces a 
$H_{\Sigma}$-algebra $\alpha_A\colon H_{\Sigma}A\to A$
defined by 
\[
\alpha(\sigma, f)=\sigma_A(f), \quad
f\in[\arity(\sigma), A].
\]
These constructions are mutually inverse, and a $\BC$-morphism
$h\colon A\to B$ is a homomorphism $A\to B$ of $\Sigma$-algebras iff
it is a homomorphism $(A, \alpha) \to (B,\beta)$ of the corresponding
$H_{\Sigma}$-algebras. We restrict ourselves to showing the latter.%
\smnote{Isn't this clear by composing the $H_\Sigma$-algebra
  homomorphism square with coproduct injections?}
To
this end, suppose first that $h\colon A\to B$ is a morphism of
$\Sigma$-algebras. Then $h\cdot\sigma_A(f)=\sigma_B(h\cdot f)$ for all
$\sigma\in\Sigma$ and all morphisms $f\colon\arity(\sigma)\to
A$. Thus, for all morphisms $f\colon\arity(\sigma)\to A$, we have
\[
  h\cdot\alpha(f)
  =
  h(\sigma_A(f))
  =
  \sigma_B(h\cdot f)
  =
  \beta(h\cdot f)
  =
  \beta\cdot H_{\Sigma}(h)(f).
\]
Conversely, if $h\colon A\to B$ is a morphism of $H_{\Sigma}$-algebras,
then $h\cdot\alpha= \beta\cdot H_{\Sigma}(h)$. To conclude
the proof, we compute
\[
  h\cdot\sigma_A(f)
  =
  h\cdot \alpha(f)
  =
  \beta\cdot H_{\Sigma}(h)(f)
  =
  \beta(h\cdot f)
  =
  \sigma_B(h\cdot f).
\] 
In short, we conclude that the assignment described above 
yields an isomorphism
\[
I\colon\AlgSigma\cong\Alg H_{\Sigma}
\]
which is moreover concrete over $\BC$: indeed, $I$ clearly 
preserves underlying sets of $\Sigma$-algebras, and it 
preserves the underlying maps of homomorphisms, as 
demonstrated above.
\end{proof}

\subparagraph*{Varieties of $\Sigma$-algebras}

\begin{lem}
  \label{L:homomorphism}
  Let $h\colon A\to B$ be a homomorphism of $\Sigma$-algebras $A,B$,
  and let $e\colon X\to A$ be a relation-preserving assignment. Then,
  for all terms $t\in\TSigma(X)$, we have:
\begin{enumerate}
\item\label{L:homomorphism-1}
$(h\cdot e)^{\#}(t)$ is defined whenever $e^{\#}(t)$ is defined,
and $(h\cdot e)^{\#}(t) = h(e^\#(t))$.
\item\label{L:homomorphism-2}
if $(h\cdot e)^\#(t)$ is defined and $h$ is an embedding, then 
$e^\#(t)$ is defined. 
\end{enumerate}
\end{lem}

\begin{proof}[Proof of \autoref{L:homomorphism}]
\begin{enumerate}
\item
  We proceed by induction on $t\in\TSigma(X)$. In case $t$ is a
  variable in context $X$, the desired statement is immediate: indeed,
  $(h\cdot e)^{\#}(t)$ \emph{is} defined, and
  \[
    (h\cdot e)^{\#}(t)
    =
    h\cdot e(t)		
    =
    h(e(t))			
    =
    h(e^\#(t))		
  \]
  where we have only used the definition of $(-)^\#$. Now, suppose
  that $t$ has the shape $\sigma(f)$ for some operation
  $\sigma\in\Sigma$ and some $f\colon |\arity(\sigma)|\to\TSigma(X)$.
  If $e^\#(\sigma(f))$ is defined, then
  \begin{enumerate}
  \item  $e^\#\cdot f(i)$ is defined for all $i\in\arity(\sigma)$;
  \item  $\alpha_A(e^\#\cdot(f\cdot g))$ for all $\alpha(g)\in\Edge(\arity(\sigma))$.
  \end{enumerate}
  Then, applying the inductive hypothesis to the subterms $f(i)$ of
  $t$, we obtain that $(h\cdot e)^\#(f(i))$ is defined for all
  $i\in\arity(\sigma)$. Moreover, by (2), and since $h$ is a
  homomorphism, we see that $\alpha_B((h\cdot e)^\#\cdot f\cdot g)$
  for all edges $\alpha(g)\in\Edge(\arity(\sigma))$. Thus,
  $(h\cdot e)^\#(t)$ is defined. To conclude the proof, we compute
  \begin{align*}
    (h\cdot e)^\#(\sigma(f))
    &= \sigma_B((h\cdot e)^\#\cdot f)  & \text{definition $(-)^\#$} \\
    &= \sigma_B(h\cdot e^\#\cdot f) & \text{induction} \\
    &= h(\sigma_A(e^\#\cdot f))	 & \text{$h$ is a homomorphism} \\
    &= h(e^\#(\sigma(f)))		& \text{definition $(-)^\#$}\tag*{\qedhere}
  \end{align*}
  
\item
Now suppose that $h$ is an embedding; we proceed by 
induction as before. For the inductive step, suppose that 
$t\in\TSigma(X)$ is of the form $t=\sigma(f)$ for some 
$\sigma\in\Sigma$ and some map 
$f\colon\arity(\sigma)\to\TSigma(X)$, and assume that 
$(h\cdot e)^\#(\sigma(f))$ is defined. Then, by definition
of $(-)^\#$, we have that 
\begin{enumerate}
\item $(h\cdot e)^\#(f(i))$ is defined for all $i\in\arity(\sigma)$ 
and 
\item $B\models\alpha((h\cdot e)^\#\cdot(f\cdot g))$ for all
	  $\alpha(g)\in\Edge(\arity(\sigma))$.
\end{enumerate}
Then, by induction, we have that $e^\#(f(i))$ is defined
for all $i\in\arity(\sigma)$. Hence, by the first part of this
proposition, we see that $(h\cdot e)^\#(f(i)) = h\cdot e^\#(f(i))$ 
for all $i\in\arity(\sigma)$. In particular, given an edge
$\alpha(g)\in\Edge(\arity(\sigma))$, we have that
$B\models\alpha(h\cdot e^\#\cdot (f\cdot g))$. Since $h$ 
is relation reflecting, it follows that 
$A\models\alpha(e^\#\cdot (f\cdot g))$ as well. Thus,
$e^\#(t)$ is defined, as desired.
\end{enumerate}
\end{proof}

\begin{proof}[Proof of \autoref{P:creation}]
  Immediate from \autoref{L:colimit-model} below.
\end{proof}
\begin{lem}\label{L:colimit-model}
  Let $(D_i\xrightarrow{d_i} A)_{i\in I}$ be a $\kappa$-directed
  colimit in $\AlgSigma$, and let $X\vdash\alpha(f)$ be a
  $\Sigma$-relation. If each $D_i$ satisfies $X\vdash\alpha(f)$, 
  then $D$ satisfies $X\vdash\alpha(f)$.
\end{lem}
\begin{proof}
  Suppose that every $D_i$ satisfies $X\vdash\alpha(f)$ and let
  $g\colon X\to A$ be a relation preserving assignment.  Using that
  $(D_i\xrightarrow{d_i} A)_{i\in I}$ is a $\kappa$-directed colimit,
  it follows, since $X$ is $\kappa$-presentable, that there exists
  $i\in I$ and a relation preserving map $\bar{g}\colon X\to D_i$ such
  that $g = d_i\cdot\bar{g}$.  Since $D_i$ satisfies
  $\Gamma\vdash\alpha(f)$, we have:
  \begin{enumerate}
  \item\label{L:colimit-model-1}
    $(\bar{g})^\#\cdot f(i)$ is defined for all $i\in\arity(\alpha)$;
  \item\label{L:colimit-model-2}
    $D_i\models\alpha((\bar{g})^\#\cdot f)$. 
  \end{enumerate}
  Applying~\autoref{L:homomorphism} to the homomorphism
  $d_i\colon D_i\to A$ and the assignment $(\bar{g})\colon X\to D_i$,
  we see that
  \[
    g^\#(f(i)) = (d_i\cdot\bar{g})^\#(f(i)) = d_i\cdot (\bar{g})^\#(f(i))
  \]
  is defined for all $i\in\arity(\alpha)$ by \ref{L:colimit-model-1}.
  Furthermore, using that $d_i$ is relation preserving, it follows
  from \ref{L:colimit-model-2} that
  $D\models\alpha(d_i\cdot(\bar{g})^\#\cdot f)$. To conclude the
  proof, we use that
  $d_i\cdot(\bar{g})^\# = (d_i\cdot\bar g)^\# = g^\#$ to see that
  $D\models\alpha(g^\#\cdot f)$. Hence $D$ satisfies
  $X\vdash\alpha(f),$ as desired.
\end{proof}

\begin{lem}\label{L:product-lemma}
Let $A=\prod_{i\in I} A_i$ be a product of algebras such
that each $A_i$ is an algebra in the variety $\V$. Then,
for all relation preserving assignments $f\colon\Gamma\to A$ 
and all terms $t\in\TSigma(\Gamma)$:  $f^\#(t)$ is defined 
iff $(\pi_i\cdot f)^\#(t)$ is defined for all $i\in I$.
\end{lem} 
\begin{proof}
To this end, we proceed by induction on 
$t\in\TSigma(\Gamma)$. If $t$ is a variable in context
$\Gamma$, then there is nothing to show. Suppose that
$t=\sigma(g)$ for some $\sigma\in\Sigma$ and some 
map $g\colon\arity(\sigma)\to\TSigma(\Gamma)$. Then,
for all $i\in I$, we have that $(\pi_i\cdot f)^\#(\sigma(g))$ is 
defined and $A_i\models\alpha((\pi_i\cdot f)^\#\cdot g)$
since each $A_i$ is an algebra in $\V$. 
\end{proof}

\begin{proof}[Proof of \autoref{P:varieties}]
It suffices to show that $\V$ is closed under
subalgebras and products in $\AlgSigma$; this
implies that $\V$ is closed under limits in $\AlgSigma$.
Since $\V$ is closed under $\kappa$-directed colimits in 
$\AlgSigma$ by~\autoref{P:creation}, we may then
apply~\autoref{T:reflection} to see that $\V$ is a reflective 
subcategory of $\AlgSigma$, as required.

We first show that $\V$ is closed under subalgebras. 
Given a subalgebra $h\colon B\hookrightarrow A$ of
an algebra $A$ in $\V$ and an axiom $\Gamma\vdash\alpha(f)$,
we verify that $B$ satisfies this axiom. To this end, let 
$e\colon\Gamma\to B$ be a relation preserving assignment. 
Using that $A$ is an algebra in $\V$, it follows that 
$(h\cdot e)^\#(f(i))$ is defined for all $i\in\Gamma$, and 
$A\models\alpha((h\cdot e)^\#\cdot f))$. Since $h$ is an
embedding, we may apply~\autoref{L:homomorphism} to 
see that $e^\#(f(i))$ is defined for all $i\in\arity(\sigma)$
hence also $(h\cdot e)^\#=h\cdot e^\#$. Using the latter, we
see that $A\models\alpha(h\cdot e^\#\cdot f)$ whence
$B\models\alpha(e^\#\cdot f)$ since $h$ is relation reflecting.
It now follows that $B$ satisfies $\Gamma\vdash\alpha(f)$,
as desired.

To conclude, we now show that $\V$ is closed under products.  Let
$A=\prod_{i\in I} A_i$ be a product of algebras in $\V$ with
projections $\pi_i\colon A\to A_i$; we are going to verify that $A$
lies in~$\V$. To this end, let $\Gamma\vdash\alpha(e)$ be an axiom of
$\V$ and let $f\colon\Gamma\to A$ be a relation preserving map.  Since
each $A_i$ is an algebra in $\V$, we have, for all $i\in I$, that
$(\pi_i\cdot f)^\#(e(j))$ is defined for all $j\in\arity(\alpha)$, and
$A_i\models\alpha((\pi_i\cdot f)^\#\cdot e)$.  By
\autoref{L:homomorphism}, we have that
$(\pi_i\cdot f)^\#=\pi_i\cdot f^\#$ since each $\pi_i$ is a
homomorphism. By \autoref{L:product-lemma}, $f^\#(e(j))$ is defined
for all $j\in\arity(\alpha)$. Then, since the $\pi_i$ are jointly
relation reflecting, we have that $A\models\alpha(f^\#\cdot e)$.
Hence $A$ satisfies $\Gamma\vdash\alpha(f)$, as desired.
\end{proof}

\begin{proof}[Proof of \autoref{T:monadic}]
  We are going use Beck's Monadicity Theorem \cite[Thm.~IV.7.1]{MacLane98}. We have seen
  in~\autoref{R:free-monad} that $V\colon\V\to\BC$
  has a left adjoint.
  %
  %
  Thus, it suffices to show that
  $V\colon\V\to\BC$ creates coequalizers of $V$-split pairs. 
  Let $f,g\colon A\to B$ be a $V$-split pair of
  homomorphisms in $\V$. That is, there exist relation preserving
  maps $c, i, j$ as depicted in the commutative diagram below
  \[
    \begin{tikzcd}
      V\!A
      \arrow[r, "V\!f", shift left]
      \arrow[r, "Vg"', shift right]
      &
      V\!B
      \arrow[l, "j", bend left=49]
      \arrow[r, "c"]
      &
      C \arrow[l, "i", bend left=49]
    \end{tikzcd}
    \qquad \text{such that}\qquad
    \begin{aligned}[t]
      c\cdot Vf &= c\cdot Vg, &  Vf\cdot j &=\id_{V\!B}, \\
      c\cdot i & = \id_C, & Vg\cdot j &=i\cdot c.
    \end{aligned}
  \]
  (This holds in $\BC$, and the equations easily imply that $c$ is an
  \emph{absolute} coequalizer of $Vf$ and
  $Vg$~\cite[Sec.~VI.6]{MacLane98}, i.e.~a coequalizer which is
  preserved by every functor; subsequently we will omit writing $V$.)
  We first show that $C$ carries the structure of a $\Sigma$-algebra:
  indeed, for each $\sigma\in\Sigma$, we define
  $\sigma_C\colon[\arity(\sigma), C]\to C$ by
  \[
    \sigma_C(m):= c\cdot\sigma_B(i\cdot m).
  \]
  Then $\sigma_C(f)$ is defined for all $f\in[\arity(\sigma), A]$
  since $i\cdot f\in[\arity(\sigma), B]$, and $\sigma_C$ is relation
  preserving because both $c$ and $\sigma_B$ are relation preserving.
  Hence $(C, (\sigma_C)_{\sigma \in \Sigma})$ defines a $\Sigma$-algebra, as claimed. We
  next show that $c\colon B\to C$ is a homomorphism: given a
  $\Gamma$-ary operation $\sigma\in\Sigma$ and a relation preserving
  map $m\colon\Gamma\to B$, we have
  \begin{align*}
    c\cdot \sigma_B(m)
    &= c\cdot\sigma_B((f\cdot j)\cdot m)
    &f\cdot j=\id_B \\
    &= c\cdot f\cdot\sigma_A(j\cdot m)
    &f \text{ a homomorphism} \\
    &= c\cdot g\cdot\sigma_A(j\cdot m)
    &c\cdot f=c \cdot g \\
    &= c\cdot\sigma_B(g\cdot j\cdot m)
    &g \text{ a homomorphism} \\
    &= c\cdot\sigma_B(i\cdot c\cdot m)
    &g\cdot j=i\cdot c \\
    &= \sigma_C(c\cdot m)
    &\text{definition of }\sigma_C
  \end{align*}
  Thus, $c$ is a homomorphism, as claimed. Moreover, the operations
  $\sigma_C$ are the unique morphisms $[\arity(\sigma), C]\to C$
  making $c$ a homomorphism.  Indeed, if $c$ is a homomorphism and
  $m\colon\arity(\sigma)\to C$ is a relation preserving map, then
  \[
    \sigma_C(m)
    =
    \sigma_C(c\cdot i\cdot m)
    =
    c\cdot\sigma_B(i\cdot m),
  \]
  where we use that $c\cdot i=\id_B$ in the first equality.

  We next show that $C$ is an algebra in $\V$. Indeed, we will prove
  that $C$ satisfies every $\Sigma$-relation satisfied by $B$; the
  claim will then follows because $B$ is an algebra in $\V$.  To this
  end, suppose that $B$ satisfies $\Gamma\vdash\alpha(e)$ and let
  $m\colon\Gamma\to C$ be a relation preserving map.  Then, the map
  $i\cdot m\colon\Gamma\to B$ is relation preserving whence, for all
  $l\in\arity(\alpha)$, $(i\cdot m)^\#(e(l))$ is defined and
  $B\models\alpha((i\cdot m)^\#\cdot e)$. Then, since $c$ is a
  homomorphism, we know by~\autoref{L:homomorphism} that, for all
  $l\in\arity(\alpha)$, $(c\cdot i\cdot m)^\#(e(l))$ is defined hence
  also $m^\#(e(l))$ is defined since $c\cdot i=\id$. Finally,
  $A\models\alpha(m\cdot e)$. Indeed, we have
  \[
    m^\#
    =
    (\id\cdot m)^\#
    =
    (c\cdot i\cdot m)^\# 
    =
    c\cdot (i\cdot m)^\#,
  \]
  and $A\models\alpha(c\cdot(i\cdot m)^\#\cdot e)$ because
  $B\models\alpha((i\cdot m)^\#)$ and $c$ is relation preserving.

  We have now seen that there is a unique $\Sigma$-algebra structure
  on $C$ making $c$ a homomorphism, and $C$ is an algebra in $\V$.
  Thus, in order to conclude the proof, it suffices to show that $c$
  is a coequalizer of $f$ and $g$ in $\V$. We already know that $c$ is
  a coequalizer in~$\BC$. Given a homomorphism $d\colon B \to D$ such
  that $d \cdot f = d \cdot g$ we therefore obtain a unique morphism
  $h\colon C \to D$ in $\BC$ such that $h \cdot c = d$. To complete
  the proof we need to show that $h$ is a homomorphism. To see this
  consider the following diagrams, for every operation symbol $\sigma$:
  \[
    \begin{tikzcd}[column sep = 40]
      {[\arity(\sigma), B]}
      \ar{r}{\sigma_B}
      \ar{d}[swap]{c \cdot (-)}
      \ar[shiftarr = {xshift=-35}]{dd}[swap]{d \cdot (-)}
      &
      B
      \ar{d}{c}
      \ar[shiftarr = {xshift=20}]{dd}{d}
      \\
      {[\arity(\sigma), C]}
      \ar{r}{\sigma_C}
      \ar{d}[swap]{h \cdot (-)}
      &
      C
      \ar{d}{h}
      \\
      {[\arity(\sigma), D]}
      \ar{r}{\sigma_D}
      &
      D
    \end{tikzcd}
  \]
  The left-hand and right-hand parts clearly commutes, and the upper
  square and outside do since $c$ and $d$ are homomorphisms. Thus, the
  desired lower square commutes when precomposed by $c\cdot (-)$. The
  latter morphism is epimorphic since it is a coequalizer being the
  image of the absolute coequalizer $c$ under the internal
  hom-functor $[\arity(\sigma),-]$. Hence, the desired
  square commutes.%
  \takeout{
  Towards this goal, let
  $d\colon B\to D$ be a homomorphism such that $d\cdot f= d\cdot g$;
  we must show that there is a unique homomorphism $d'\colon C\to D$
  such that $d'\cdot c=d$.  Indeed, consider $d':=d\cdot i$. Then
  \[
    d'\cdot c
    =
    d\cdot i\cdot c
    =
    d\cdot g \cdot j
    =
    d\cdot f \cdot j
    =
    d,
  \]
  as required. As for uniqueness, note that if $d'\cdot c = d$, then 
  \begin{align*}
    d'
    &= d\cdot i\cdot c
    &\text{definition of $d'$}
    \\
    &= d\cdot g\cdot j
    &i\cdot c=g\cdot j
    \\
    &= d\cdot f\cdot j
    & d\cdot f=d\cdot g
    \\
    &= d
    &f\cdot j=\id_B
  \end{align*}
  Finally, $d'\colon C\to D$ is a homomorphism because for all
  $\sigma\in\Sigma$ and all $m\colon\arity(\sigma)\to C$ we have
  \begin{align*}
    d'(\sigma_C(m))
    &= d'\cdot c(\sigma_B(i\cdot m))
    &\text{definition of $\sigma_C$}
    \\
    &= d(\sigma_B(i\cdot m))
    & d'\cdot c= d
    \\
    &= \sigma_D(d\cdot i\cdot m)
    & \text{$d$ is a homomorphism}
    \\
    &= \sigma_D(d'\cdot m)
    &\text{definition of $d'$}.
  \end{align*}
  We conclude that $c$ is a coequalizer in $\V$, as desired.}
\end{proof}

\begin{proof}[Proof of~\autoref{C:monadic}]
We know that the assignment $\V\mapsto\bbT_{\V}$ of a 
$\kappa$-ary relational algebraic theory to its free-algebra 
monad $\bbT$ preserves categories of models by
\autoref{T:monadic}, and $\bbT_{\V}$ is $\kappa$-accessible, as
discussed in~\autoref{R:free-monad}. To conclude, we must
show that $\bbT_{\V}$ is enriched. That is, given $\alpha\in\Rels$ 
and $(f_i\colon X\to Y)_{i\in\arity(\alpha)}$ such that 
$[X, Y]\models\alpha(f_i)$, we must show that
$[T_{\V}X, T_{\V}Y]\models\alpha(T_{\V}f_i)$.

To this end, let $\alpha\in\Rels$ and $f_i\colon X\to Y$ be as
above, and let $e\colon E\hookrightarrow T_{\V}X$ be the 
substructure of $\T_{\V}X$ given by all $t\in T_{\V}X$ such that 
$T_{\V}Y\models\alpha(T_{\V} f_i(t))$. We first show that 
$\eta_X(x)\in E$ for all $x\in X$. We know that
$[X, Y]\models\alpha(f_i)$, and this means that $Y\models\alpha(f_i(x))$
for all $x\in X$. Since $\eta_Y\colon Y\to T_{\V}Y$ is relation
preserving, it follows that $T_{\V}Y\models\alpha(\eta_Y\cdot f_i(x))$.
By naturality of $\eta$, the square below is commutative
\smnote{I think that natural transformations are mostly written horizontally.}          
\[
  \begin{tikzcd}
    X \ar{r}{\eta_X}
    \ar{d}[swap]{f_i}
    &
    T_\V X
    \ar{d}{T_\V f_i}
    \\
    Y
    \ar{r}{\eta_Y}
    &
    T_\V Y
  \end{tikzcd} 
\]
whence $T_{\V}Y\models\alpha( T_{\V} f_i(\eta_X(x)))$. That is,
$\eta_X(x)\in E$, as claimed.

Furthermore, $E$ is closed under the operations of the algebra
$T_{\V}X$ because each operation
$\sigma_{T_{\V}X}\colon [\arity(\sigma), T_{\V}X]\to T_{\V}X$ is
relation preserving. Indeed, let $\sigma\in\Sigma$ and let
$h\colon\arity(\sigma)\to T_{\V}X$ be a relation preserving map such
that $h[\arity(\sigma)]\subseteq E$; we are going to verify that
$\sigma_{T_{\V}X}(h)\in E$. Note that we have the commutative square
\[
  \begin{tikzcd}[column sep = 40]
    {[\arity(\sigma), T_{\V}X]}
    \arrow[r, "\sigma_{T_{\V}X}"]
    \arrow[d, "T_{\V}f_i\cdot(-)"']
    &
    T_{\V}X \arrow[d, "T_{\V}f_i"]
    \\
    {[\arity(\sigma), T_{\V}Y]}
    \arrow[r, "\sigma_{T_{\V}Y}"]
    & T_{\V}Y                       
  \end{tikzcd}
\]
Thus, for the map $h\colon\arity(\sigma)\to T_{\V}X$, we have 
that
\[
  T_{\V}f_i(\sigma_{T_{\V}X}(h)) = \sigma_Y(T_{\V}f_i\cdot h)
  \qquad
  \text{for every $i\in\arity(\alpha)$}.
\]
Now, using that $h[\arity(\sigma)]\subseteq E$, it follows that 
$T_{\V}Y\models \alpha(T_{\V}f_i(h(j)))$ for 
all $j\in\arity(\sigma)$. Thus, by definition of 
$\Edge([\arity(\sigma), T_{\V}Y])$, we have that 
$[\arity(\sigma), T_{V}Y]\models\alpha(T_{\V}f_i\cdot h)$. This
implies that $T_{\V}Y\models\alpha(\sigma_{T_{\V}Y}(T_{\V}f_i\cdot h))$
since $\sigma_{T_{\V}Y}$ is relation preserving. Applying the 
commutative square above, it follows that 
$T_{\V}Y\models\alpha(T_{\V}f_i\cdot\sigma_{T_{\V}X}(h))$
hence also $\sigma_{T_{\V}X}(h)\in E$, as claimed.

We conclude that $E$ is a $\Sigma$-subalgebra of $T_{\V}X$
containing $\eta_X[X]$. Since $T_{\V}X$ is a free algebra of $\V$
with the universal morphism $\eta_X$, it follows that $E=T_{\V}X$. In
particular, $T_\V Y \models \alpha(T_{\V}f_i(t))$ holds for all $t \in
T_\V X$, whence $\alpha(T_\V f_i)$ holds in $[T_{\V}X, T_{\V}Y]$, as
desired. 
\end{proof}

\subparagraph*{Relational Logic} 

\begin{proof}[Proof of~\autoref{L:admissible}]
\textbf{Admissibility of ($\mathsf{Arity}$)}:
Admissibility of the rule $(\mathsf{Arity})$ is immediate: in any derivation
of $X\vdash\isdef\sigma(m)$, the last rule applied was $(\EAr)$. We may
apply this rule just in case we have derivations of all desired relational
judgements $X\vdash\alpha(f)$.

\smallskip\noindent \textbf{Admissibility of ($\mathsf{Subterm}$)}: We
proceed to show that whenever $X\vdash\alpha(f)$ is derivable, where
$f\colon\arity(\alpha)\to\TSigma(X)$, then $X\vdash\isdef u$ is
derivable for all $u\in\sub(f)$; we do so by making a case distinction
on the basis of whether $u\in\sub(f)$ is a variable (in $X$) or
$u=\sigma(m)$ for some $\sigma\in\Sigma$ and some map
$m\colon|\arity(\sigma)|\to\TSigma(X)$. If $u$ is a variable, then
$X\vdash\isdef u$ is derivable via $(\mathsf{Var})$. So let us suppose
that $u=\sigma(m)$. Then, by application of $(\mathsf{Arity})$ (which
was shown to be admissible above), we have a derivation of
$X\vdash\beta(m\cdot g)$ for all edges $\beta(g)$ such that
$\arity(\sigma)\models\beta(g)$. Thus, by application of $(\EAr)$, we
have a derivation of $X\vdash\isdef\sigma(m)$, as desired.
\end{proof}

\begin{cor}\label{C:admissible}
Let $\sigma\in\Sigma$ and let $m\colon|\arity(\sigma)|\to\TSigma(\Gamma)$. 
Then $\Gamma\vdash\isdef\sigma(m)$ iff $\Gamma\vdash\alpha(m\cdot f)$
is derivable for all edges $\alpha(f)$ such that $\arity(\sigma)\models\alpha(f)$.
\end{cor}

\begin{proof}
Suppose that $\Gamma\vdash\isdef\sigma(m)$ is derivable. Then 
$\Gamma\vdash\alpha(m\cdot f)$ is derivable for all edges $\alpha(f)$
in $|\arity(\sigma)|$ such that $\arity(\sigma)\models\alpha(f)$ by
application of $(\mathsf{Arity})$; this rule is admissible by~\autoref{L:admissible}. 
Conversely, if $\Gamma\vdash\alpha(m\cdot f)$ is derivable 
for all edges $\alpha(f)$ such that $\arity(\sigma)\models\alpha(f)$,
then $\Gamma\vdash\isdef\sigma(m)$ is derivable via ($\EAr$),
as desired. 
\end{proof}

In~\textbf{\autoref{R:relational-logic}} we have suggested that the relational
logic enjoys an admissible substitution rule. We substantiate this claim now: 

\begin{propn}\label{P:substitution}
  Let $X,Y\in\BC$ and let $\tau\colon Y \to \TSigma(X)$ be a
  substitution. The following rules are admissible:
  \[
    (\mathsf{SubsR})\;\frac{\{X\vdash\alpha(\tau\cdot f)\mid
      Y\models\alpha(f)\}\cup\{X\vdash\isdef\tau(y)\mid y\in Y\}
      \quad 
      Y\vdash\beta(g)}{X\vdash\beta(\tau\cdot g)}~%
    \begin{array}{l@{}l}
      (&\beta\in\Rels \\
      &g\colon\arity(\beta)\to\TSigma(Y))
    \end{array}
  \]
  \[
    (\mathsf{SubsD})\;\frac{\{X\vdash\alpha(\tau\cdot f)\mid
      Y\models\alpha(f)\}\cup\{X\vdash\isdef\tau(y)\mid y\in Y\}
      \quad 
      Y\vdash \isdef t}{X\vdash\isdef \tau(t)}~(t\in\TSigma(Y))		
  \]
\end{propn}
\begin{proof}
  Suppose that all premisses in
  $\{X\vdash\alpha(\tau\cdot f)\mid
  Y\models\alpha(f)\}\cup\{X\vdash\isdef\tau(y)\mid y\in Y\}$ are
  derivable. We prove by simultaneous induction on derivations that
  whenever $Y\vdash\beta(g)$ is derivable, then
  $X\vdash\beta(\tau\cdot g)$ is derivable, and whenever
  $Y\vdash\isdef t$ is derivable, then $X\vdash\isdef\tau(t)$ is
  derivable.

  As a base case, observe that if $Y\vdash\beta(g)$ was derived using
  $(\mathsf{Ctx})$, then $Y\models\beta(g)$ whence
  $X\vdash\beta(\tau\cdot g)$ is derivable by assumption. Similarly,
  if $Y\models\isdef y$ was derived using $(\mathsf{Var})$, then
  $X\vdash\isdef\tau(y)$ is derivable by assumption.
  As for the inductive step, we first note that corresponding to $(\IAr)$ is analogous to
  the case for $(\mathsf{Ax})$ shown further below; the remaining cases are as
  follows.
  \begin{enumerate}
  \item[$(\mathsf{RelAx})$] Suppose that $Y \vdash \beta(g)$, where $g
    = \tau' \cdot f$, $\Phi \implies \beta(f)$ is some relational
    axiom and $\tau'\colon \Vars \to \TSigma(Y)$, is was derived in
    the last step from premisses:
    \begin{enumerate}
    \item $X \vdash \tau'\cdot \varphi$ for all $\varphi \in \Phi$;
    \item $\isdef\tau'(f(i))$ for all $i \in \arity(\alpha)$.
    \end{enumerate}
    By the induction hypothesis, $Y\vdash \tau \cdot \tau'
    \cdot\varphi$ for $\varphi \in \Phi$ and $Y \dashv
    \isdef\tau(\tau'(f(i)))$ for $i \in \arity(\alpha)$ are
    derivable. Thus $Y \vdash \beta(\tau \cdot \tau' \cdot f)$ is
    derivable by an application of the $(\mathsf{RelAx})$ rule.
    
  \item[$(\EAr)$:] Suppose that $Y\vdash\isdef \sigma(f)$, for
    $\sigma\in\Sigma$ and $f\colon\arity(\sigma)\to\TSigma(Y)$, was
    derived in the last step from premisses $Y\vdash\alpha(f\cdot g)$
    for all $\alpha(g)\in\arity(\sigma)$ and $X\vdash \isdef f(i)$ for
    all $i \in \arity(\alpha)$. By induction,
    $X\vdash\alpha(\tau\cdot f\cdot g)$ is derivable for all
    $\alpha(g)\in\arity(\sigma)$, and $X\vdash\isdef(\tau\cdot f)(i)$ is
    derivable for all $i\in\arity(\sigma)$. Hence
    $X\vdash\isdef\sigma(\tau \cdot f)$ is derivable via $(\EAr)$, and
    since $\tau(\sigma(f)) = \sigma( \tau \cdot f)$ we are done. 

  \item[$(\mathsf{Mor})$:] Suppose that $Y\vdash\beta(g)$, where
    $g\colon\arity(\beta)\to\TSigma(Y)$ has the form of an assignment
    $g(i) = \sigma(f_i)$ for some
    $\sigma\in\Sigma$ and $f_i\colon \arity(\sigma) \to \TSigma(Y)$
    for $i \in \arity(\beta)$,
    was derived  from premisses
    \begin{enumerate}
    \item $Y\vdash\beta(f_i(j))$ for all $j\in\arity(\sigma)$;
    \item $Y\vdash\isdef\sigma(f_i)$ for all $i\in\arity(\beta)$;
    \end{enumerate}
    Applying the inductive hypothesis to these items, it follows immediately
    that $X\vdash\beta(\tau\cdot g)$ is derivable via $(\mathsf{Mor})$.

  \item[$(\mathsf{Ax})$:] Finally, suppose that
    $Y\vdash\beta(\kappa\cdot k)$ was derived via $(\mathsf{Ax})$
    using an axiom $\Gamma\vdash\beta(k)\in\E$ and a substitution
    $\kappa\colon\Gamma\to\TSigma(Y)$ from premisses
    $Y\vdash\alpha(\kappa\cdot f)$ for all $\alpha(f)\in\Edge(\Gamma)$
    and $Y\vdash\isdef\kappa(z)$ for all
    $z\in\Gamma$. 
    By induction, $X\vdash\alpha(\tau\cdot\kappa\cdot f)$ is derivable
    for all $\alpha(f)\in\Edge(\Gamma)$, and
    $X\vdash\isdef(\tau\cdot\kappa)(z)$ is derivable for all
    $z\in\Gamma$. Thus, by application of $(\mathsf{Ax})$, we have a
    derivation of $X\vdash\beta(\tau\cdot\kappa\cdot k)$.  \qedhere
  \end{enumerate}
\end{proof}
\subparagraph*{Constructing free algebras}
\begin{lem}\label{L:ind}
  \begin{enumerate}
  \item\label{L:ind:1} The relational structure on $\F X$ is independent of the
    choice of $u$.

  \item\label{L:ind:2} $\sigma_{\F X}$ is defined on
    $[\arity(\sigma),\Term_{\V}(X)/{\sim}]$ and independent of the
    choice of the splitting $u$.
  \end{enumerate}
\end{lem}
\begin{proof}
  \begin{enumerate}
  \item Given another splitting $v$ of $q$, we need to prove that
    $\Term_\V(X) \models \alpha(u \cdot f)$ iff
    $\Term_\V(X)\models \alpha(v \cdot f)$ for every
    $f\colon \arity(\alpha) \to \F X$. Equivalently,
    $X \vdash \alpha(u \cdot f)$ is derivable iff so is
    $X \vdash \alpha(v \cdot f)$. Since $u, v$ are splittings of $q$
    we have that for each $i \in \arity(\alpha)$,
    $X \vdash \varphi(u \cdot g(i), v \cdot g(i))$ is derivable for
    every $\varphi(x,y) \in \Eq(x,y)$. Since by
    \autoref{assn:equality}, $\RelAx$ explicitly includes axioms
    stating that all relations are closed under $\Eq(-,-)$, the claim
    follows by an application of the $(\textsf{RelAx})$ rule in each
    direction.%
    
  \item For the first claim, we show that $\sigma(u\cdot g)$ is
    defined for all $g\in[\arity(\sigma), \F X]$. Since $g$ is
    relation preserving, it follows that for every edge $\alpha(f)$
    such that $\arity(\sigma)\models\alpha(f)$ we have
    $\F X\models\alpha(u\cdot g\cdot f)$. That is,
    $X\vdash\alpha(u\cdot g\cdot f)$ is derivable.  Hence
    $X\vdash\isdef\sigma(u\cdot g)$ is derivable via $(\EAr)$, which
    implies $\sigma(u\cdot g)\in\Term_{\V}(X)$.

    We now show that $[\sigma(u\cdot g)]$ is independent of the choice
    of $u$. To this end, suppose that we are given another splitting
    $v$, and let $\varphi(x,y)\in\Eq(x,y)$.  By definition of~$\sim$,
    it suffices to show that
    $X\vdash\varphi(\sigma(u\cdot g), \sigma(v\cdot g))$ is
    derivable. Since $u,v$ are splittings of~$q$, we have that for
    each $i\in\arity(\sigma)$,
    $X\vdash\varphi(u\cdot g(i),v\cdot g(i))$ is derivable. The stated
    goal $X\vdash\varphi(\sigma(u\cdot g), \sigma(v\cdot g))$ follows
    by rule $(\mathsf{Mor})$.\qedhere
  \end{enumerate}
\end{proof}
\begin{rem}
  The set $\Term_\V(X)$ is a $\Sigma$-algebra under the operations
  given by the usual term formation: for every operation symbol
  $\sigma$ in $\Sigma$ we have
  \[
    \sigma_{\Term_\V(X)}\colon [\arity(\sigma), \Term_\V(X)] \to
    \Term_{\V}(X)
    \qquad\text{given by}\qquad
    f \mapsto \sigma(f). 
  \]
  Indeed, for a relation preserving map $f\colon \arity(\sigma) \to
  \Term_\V(X)$ we have $X \vdash \alpha(f \cdot g)$  for every
  $\alpha(g) \in \arity(\sigma)$. Thus $X \vdash \isdef{\sigma(f)}$ by
  an application of the $(\textsf{E-Ar})$ rule. 
\end{rem}
\begin{cor}\label{C:homo}
  The quotient map $q\colon \Term_\V(X) \to \F X$ is a relation
  preserving $\Sigma$-algebra homomorphism. 
\end{cor}
\begin{proof}
  To see that $q$ relation preserving, suppose
  $\T_\V(X) \models \alpha(f)$ for
  $f\colon \arity(\alpha) \to \Term_V(X)$. By \autoref{L:ind}\ref{L:ind:1}, we may
  assume that the spliting $u$ is chosen such that
  $u \cdot q \cdot f = f$. Therefore $\Term_V(X) \models \alpha(u
  \cdot q \cdot f)$ which gives $\F X \models \alpha(q \cdot f)$ as
  desired. 

  Given $\sigma$ in $\Sigma$ and $g\colon \arity(\sigma) \to
  \Term_\V(X)$ relation preserving, we may assume by \autoref{L:ind}\ref{L:ind:2}
  that the splitting $u\colon \F X \to \Term_\V(X)$ is chosen such
  that $u \cdot q \cdot g = g$. Then we have
  \begin{align*}
    q \cdot \sigma_{\Term_\V(X)}(g)
    &=
    q(\sigma(g))
    &
    \text{def.~of $\sigma_{\Term_{\V}(X)}$}
    \\
    &=
    q(\sigma(u \cdot q \cdot g))
    &
    \text{since $g = u\cdot q \cdot g$}
    \\
    &=
    \sigma_{\F X}(q \cdot g).
    &\text{def.~of $\sigma_{\F X}$} & \qedhere
  \end{align*}
\end{proof}

%
%
\begin{proof}[Proof of \autoref{thm:free-alg}]
  \takeout{
    \textbf{$\F X$ is an algebra in $\V$}: Indeed, let
    $\Delta\vdash\beta(g)$ be an axiom of $\V$, where
    $g\colon \arity(\beta) \to T_\Sigma(\Delta)$ and let
    $e\colon\Delta\to\F X$ be a relation preserving map; we are going to
    show that $e^\#(g(i))$ is defined for all $i\in\arity(\beta)$ and
    that $\F X\models\beta(e^\#\cdot g)$. (This implies that $\F X$
    satisfies $\Delta\vdash\beta(g)$, as required.) To this end, define
    a substitution $\tau\colon\Delta\to\Term_\V(X)$ (again note
    $\Term_\V(X)\subseteq T_\Sigma(X)$) as the map
    \[
      \tau
      =
      |\Delta|\xra{e}\F X \xra{u}\Term_{\V}(X)
    \]
    
    For all edges $\alpha(f)\in\Delta$, we have that
    $\F X\models\alpha(e\cdot f)$ because $e$ is
    relation-preserving. Equivalently,
    $\Term_V(X) \models \alpha(u \cdot e \cdot f)$, which means that
    $X\vdash\alpha(\tau\cdot f)$ is derivable. Hence also
    $X\vdash\beta(\tau\cdot g)$ is derivable via $(\mathsf{Ax})$.  Then,
    using the $(\mathsf{Subterm})$ rule (\autoref{L:admissible}), we see
    that for each $i\in\arity(\beta)$, $X\vdash\isdef(\tau\cdot g(i))$
    is derivable.  By \autoref{lem:FX-sondness}, and because
    $q\cdot\tau=e$, we obtain that $e^\#(f(i))$ is defined in $\F X$ and
    $\F X\models\beta(e^\#\cdot g)$, as desired.

    \smnote[inline]{Alternative proof attempt:}}
  Indeed, let $\Delta\vdash\beta(g)$ be an axiom of $\V$, where
  $g\colon \arity(\beta) \to T_\Sigma(\Delta)$, and let
  $e\colon\Delta\to\F X$ be a relation preserving map; we are going to
  show that $e^\#(g(i))$ is defined for all $i\in\arity(\beta)$ and
  that $\F X\models\beta(e^\#\cdot g)$. (This implies that $\F X$
  satisfies $\Delta\vdash\beta(g)$, as required.) To this end, define
  a substitution $\tau\colon\Delta\to\Term_\V(X)$ (again note
  $\Term_\V(X)\subseteq T_\Sigma(X)$) as the map
  \[
    \tau
    =
    |\Delta|\xra{e}\F X \xra{u}\Term_{\V}(X).
  \]
  For all edges $\alpha(f)\in\Delta$, we have that
  $\F X\models\alpha(e\cdot f)$ because $e$ is
  relation-preserving. Equivalently,
  $\Term_V(X) \models \alpha(u \cdot e \cdot f)$, which means that
  $X\vdash\alpha(\tau\cdot f)$ is derivable. Hence also
  $X\vdash\beta(\tau\cdot g)$ is derivable via $(\mathsf{Ax})$. Since
  $q$ is relation preserving by \autoref{C:homo} we have
  $\F X \models \beta(q \cdot \tau \cdot g)$. We shall show below that
  for all $r$ in $\sub(g)$ we have that both $q \cdot \tau(r)$ and $e^\#(r)$ are
  defined and 
  \begin{equation}\label{eq:qtau}
    q \cdot \tau(r) = e^\#(r).
  \end{equation}
  (Recall from \autoref{def:subst} that we write $\tau$ also for the
  extension $\bar\tau\colon T_\Sigma(\Delta) \to T_\Sigma(X)$.) Thus,
  in particular $e^\#(g(i))$ is defined for all $i \in \arity(\beta)$,
  and we obtain $\F X \models \beta(e^\# \cdot g)$ as desired.

 To prove~\eqref{eq:qtau} we proceed by induction on $r$. For a
  variable $x$ in $\sub(g)$ we clearly have that both $q \cdot \tau$
  and $e^\#$ are defined in $x$ and
  \[
    q\cdot \tau(x) = q \cdot u \cdot e(x) = e(x) = e^\#(x),
  \]
  using the definitions of $\tau$ and $e^\#$ and that
  $q \cdot u = \id$. For a term $\sigma(k)$ in $\sub(g)$ we know by
  induction that both $q \cdot \tau(k(i))$ and $e^\#(k(i))$ are
  defined and $q\cdot \tau(k(i)) = e^\#(k(i))$ for all
  $i \in \arity(\sigma)$.  Then $\tau(\sigma(k))$ is a provably
  defined term, i.e.~it lies in $\Term_\V(X)$, so that $q$ is defined
  on it; indeed, since $X \vdash \beta(\tau \cdot g)$ is derivable,
  this follows from an application of the $(\textsf{Subterm})$ rule in
  \autoref{L:admissible} since $\sigma(k)$ in $\sub(g)$ yields that
  $\tau(\sigma(k))$ lies in $\sub(\tau \cdot g)$. Now we obtain
  \begin{align*}
    q \cdot \tau(\sigma(k))
    &=
    q(\sigma(\tau \cdot k))
    &
    \text{def.~of $\tau = \bar \tau$}
    \\
    &=
    \sigma_{\F X}(q\cdot \tau \cdot k)
    &\text{by \autoref{C:homo}}
    \\
    &=
    \sigma_{\F X}(e^\# \cdot k)
    & \text{by induction}
    \\
    &=
    e^\#(\sigma(k))
    &
    \text{def.~of $e^\#$}.
  \end{align*}

  \smallskip\noindent \textbf{$\F X$ is free on $X$}: We define
  $\eta_X\colon X \to \F X$ by $\eta_X(x) = [x]$ and prove that this
  is a universal map. First observe that $\eta$ is clearly relation
  preserving since for every $X \models \alpha(f)$ we have
  $\F X \models \alpha(\eta_X \cdot f)$ by the rule
  $(\mathsf{Ctx})$. Now given $f\colon X \to A$ relation preserving we
  define $\bar f\colon \F X \to A$ by
  \[
    \bar f([t]) = f^\#(t)
    \qquad
    \text{for all $t \in \Term_{\V}(X)$}.
  \]
  Then $\bar f$ is well-defined by soundness (proved below), and we 
  clearly have $\bar f \cdot \eta_X = f$.
  \begin{enumerate}
  \item We prove that $\bar f$ is relation preserving. Suppose that
    $\F X \models \alpha(g)$. Equivalently we have
    $\Term_\V(X) \models \alpha(u \cdot g)$, which in turn means that
    $X \vdash \alpha(u \cdot g)$ is derivable. By soundness, $A$
    satisfies $X \vdash \alpha(u \cdot g)$. In particular, we have
    $A \models \alpha(f^\# \cdot u \cdot g)$. Then we obtain
    $A \models \alpha(\bar f \cdot g)$ as desired since
    \[
      f^\# \cdot u = \bar f \cdot q \cdot u = \bar f
    \]
    because $f^\# \cdot q = \bar f$ by the definition of $\bar f$.

  \item Next we show that $\bar f$ is $\Sigma$-algebra homomorphism.
    Indeed, for every operation symbol $\sigma$ we show that the
    following square commutes:
    \[
      \begin{tikzcd}[column sep = 40]
        {[\arity(\sigma), \F X]}
        \ar{r}{\sigma_{\F X}}
        \ar{d}[swap]{\bar f \cdot (-)}
        &
        \F X
        \ar{d}{\bar f}
        \\
        {[\arity(\sigma), A]}
        \ar{r}{\sigma_A}
        &
        A
      \end{tikzcd}
    \]
    Indeed, for every $g$ in $[\arity(\sigma), \F X]$ we have
    \begin{align*}
      \bar f (\sigma_{\F X}(g))
      &= \bar f([\sigma(u \cdot g)])
      &
      \text{def.~of $\sigma_{\F X}$}
      \\
      &=
      f^\#(\sigma(u \cdot g))
      &\text{def.~of $\bar f$}
      \\
      &=
      \sigma_A(f^\# \cdot u \cdot g)
      & \text{def.~of $f^\#$}
      \\
      &= \sigma_A(\bar f \cdot q \cdot u \cdot g),
      &
      \text{since $\bar f \cdot q = f^\#$ by def.}
      \\
      &=
      \sigma_A(\bar f \cdot g)
      &
      \text{since $q \cdot u = \id$.}
    \end{align*}
    
  \item Finally, we show that $\bar f$ is unique. Indeed this follows
    from the required homomorphism property and the required extension
    property $\bar f\cdot \eta_X=f$, since by construction every
    element of $\F X$ is denoted by a well-founded term formed from
    the elements of~$X$ using the operations of~$\Sigma$ (formally,
    use well-founded induction on terms).
    \qedhere
  \end{enumerate}
\end{proof}

\begin{lem}\label{L:complete}
  Let $t\in\TSigma(X)$. Then, for the universal map
  $\eta_X\colon X\to\F X$, if $\eta_X^\#(t)$ is defined, then
  $t\in\Term_{\V}(X)$ and $\eta_X^\#(t) = [t]$.
\end{lem}
\noindent
The latter means that $\eta^\#$ is the canonical quotient map
$q\colon \Term_\V(X) \to \F X$.
\begin{proof}
By induction on terms. If $t$ is a variable in $X$,
then $X\vdash\isdef t$ is derivable via $(\mathsf{Var})$
and $\eta_X^\#(t)=\eta_X(t)=[t]$. Inductively assume
that the lemma holds for all $s\in\sub(t)$, and let
$t=\sigma(m)$ for some $\sigma\in\Sigma$ and 
some $m\colon\arity(\sigma)\to\TSigma(X)$ such that
$\eta_X^\#(\sigma(m))$ is defined. Then:
\begin{enumerate}
\item
$\eta_X^\#(m(i))$ is defined for all $i\in\arity(\sigma)$;
\item
$\F X\models\beta(\eta_X^\#\cdot m\cdot g)$ for all 
$g\colon\arity(\beta)\to\arity(\sigma)$ such that 
$\arity(\sigma)\models\beta(g)$.
\end{enumerate}
Combining the inductive hypothesis with (1), we 
have that $X\vdash\isdef m(i)$ for all $i\in\arity(\sigma)$, 
and $\eta^\#(m(i)) = [m(i)]$; that is, $\eta_X^\# \cdot m = q \cdot m$.  

In particular, by (2), we have that
$\Term_\V \models \beta(u\cdot\eta_X^\#\cdot m\cdot g)$ for all
$g\colon\arity(\beta)\to\arity(\sigma)$ such that
$\arity(\sigma)\models\beta(g)$ and where
$u\colon \F X \to \Term_\V(X)$ is a splitting of the canonical
quotient map~$q$. By \autoref{L:ind}\ref{L:ind:2}, we may assume that
$u$ us chosen such that $u \cdot q \cdot m = m$. Now, we have%
\smnote{I think element-free is better.}
\[
  u\cdot\eta^\#_X\cdot m\cdot g
  =
  u \cdot q \cdot m \cdot g
  =
  m \cdot g,
\]
whence $\Term_\V \models \beta(m \cdot g)$. This means that
$X\vdash\beta(m\cdot g)$ is derivable for all
$\beta(g)\in\arity(\sigma)$. Hence $X\vdash\isdef\sigma(m)$
is derivable via $(\EAr)$, which means that $\sigma(m)$ lies in
$\Term_\V(X)$. Moreover, we have
\begin{align*}
  \eta_X^\#(\sigma(m))
  &= \sigma_{\F X} (\eta_X^\# \cdot m)
  & \text{def.~of $\eta_X^\#$}
  \\
  &= \sigma_{\F X} (q \cdot m)
  &\text{induction hypothesis}
  \\
  &= q(\sigma(m))
  &\text{by \autoref{C:homo}}
  \\
  &= [\sigma(m)].\tag*{\qedhere}
\end{align*}
%
%
\end{proof}

\takeout{ 
\begin{lem}[Substitution lemma]\label{L:substitution}
  Let $A$ be a $\Sigma$-algebra, let $t\in T_\Sigma(X)$, let
  $\tau\colon |X|\to T_\Sigma(Y)$ be a substitution, and let
  $e\colon |Y|\to A$ be an assignment such that $e^\#(\tau(t))$ is
  defined. Then $e^\#(\tau(s))$ is defined for all~$s\in\sub(t)$, and
  \begin{equation*}
    e^\#(\tau(s)) =e_\tau^\#(s)\quad\text{where $e_\tau(x)=e^\#(\tau(x))$ for $x\in X\cap\sub(t)$}
  \end{equation*}
  (and $e_\tau(x)$ is defined arbitrarily for
  $x\in X\setminus\sub(t)$).
\end{lem}
\begin{proof}
  Induction on subterms~$s$ of~$t$. We note first that by the
  definition of evaluation, definedness is inherited by subterms, so
  $e^\#(\tau(s))$ is defined for all such~$s$, in particular, for~$s$
  a variable in~$X$ as claimed in the statement.

  First let~$s$ be a variable~$x$. Then
  \begin{equation*}
    e^\#(\tau(x)) = e_\tau(x) = e_\tau^\#(x).
  \end{equation*}
  Now, let $s=\sigma(f)$. Then
  \begin{align*}
    e^\#(\tau(\sigma(f)))
    & = e^\#(\sigma(\tau\cdot f))
    &\text{def.~of $\tau = \bar \tau$}
    \\
    & = \sigma_A(e^\#\cdot\tau\cdot f)
    &\text{def.~of $e^\#$}
    \\
    & = \sigma_A(e_\tau^\#\cdot f)
    &\text{induction hypothesis}\\
    & =e_\tau^\#(\sigma(f)) &
    \text{def.~of $e^\#_\tau$}. &\qedhere
  \end{align*}
\end{proof}
} 

\begin{notn}
  Let $A$ be a $\Sigma$-algebra, let $\tau\colon X\to\TSigma(Y)$ be a
  substitution, and let $e\colon X\to A$ be an assignment such that
  $e^\#\cdot\tau(x)$ is defined for all~$x\in X$. We define
  \[
    e_{\tau} = \big(X\xra{\tau}\TSigma(Y)\xra{e^\#} A\big).
  \]
\end{notn}

\begin{lem}[Substitution lemma]\label{L:subs}
  Let $A$ be a $\Sigma$-algebra in $\V$, let
  $\tau\colon X\to\TSigma(Y)$ be a substitution, and let
  $e\colon Y\to A$ be an assignment such that $e^\#(\tau(x))$ is
  defined for all $x\in X$. Then, for all $t\in\TSigma(X)$:
  \[
    \text{$e^\#(\overline{\tau}(t))$ is defined}
    \qquad\text{iff}\qquad
    \text{$e_{\tau}^\#(t)$ is defined;}
  \]
  in this case, $e^\#(\overline{\tau}(t))=e_{\tau}^\#(t)$.
\end{lem}
\begin{proof}
  By induction on terms. If $t=x\in X$, then
  \begin{align*}
    e^\#\cdot\overline{\tau}(t)
    &= e^\#\cdot\tau(x)
    &\text{def.~of $\overline\tau$} \\
    &=e_{\tau}(x)
    &\text{def.~of $e_{\tau}$} \\
    &=e_{\tau}^\#(t)
    &\text{def.~of $e_\tau^\#$},
  \end{align*}
  with the reading of the equational steps including that the left
  side is defined iff the right side is.

  Next, suppose that $t=\sigma(f)$ for some $\sigma\in\Sigma$ and some
  map $f\colon|\arity(\sigma)|\to\TSigma(X)$. We want to show that
  $e^\#(\overline{\tau}(\sigma(f))) = e^\#(\sigma(\overline{\tau}\cdot
  f))$ is defined iff $e_{\tau}^\#(\sigma(f))$ is defined.

  Observe that if $e^\#(\sigma(\overline{\tau}\cdot f))$ is defined,
  then by definition of $e^\#$ we have
  \begin{enumerate}
  \item $e^\#(\overline{\tau}(f(i)))$ is defined for all $i\in\arity(\sigma)$; 
  \item $A\models\alpha(e^\#\cdot\overline{\tau}\cdot (f\cdot g))$ for
    all $\alpha(g)\in\arity(\sigma)$.
  \end{enumerate}
  Thus, by the inductive hypothesis, we have that $e^\#_{\tau}(f(i))$
  is defined for all $i\in\arity(\sigma)$, and
  $A\models\alpha(e_{\tau}^\#\cdot (f\cdot g))$ for all
  $\alpha(g)\in\arity(\sigma)$.
  Thus, $e_{\tau}^\#(\sigma(f))$ is defined. The converse direction
  follows completely analogously (indeed, simply note that each
  implication is reversible). We conclude the proof with a computation
  in which all expressions are equi-defined:
  \begin{align*}
    e^\#(\overline{\tau}(\sigma(f)))
    &= e^\#(\sigma(\overline{\tau}\cdot f))			
    &\text{def.~of $\overline{\tau}$} 
    \\
    &= \sigma_A(e^\#\cdot\overline{\tau}\cdot f)         
    &\text{def.~of $e^\#$}
    \\
    &= \sigma_A(e_{\tau}^\#\cdot f)				
    &\text{induction hypothesis} 
    \\
    &= e_{\tau}^\#(\sigma(f))					
    &\text{def.~of $e_{\tau}^\#$} \tag*{\qedhere}
  \end{align*}
\end{proof}
\begin{cor}\label{C:subs}
  Let $A, \tau,$ and $e$ be as in \autoref{L:subs} above. Then, for all
  $\alpha\in\Rels$ and all maps $f\colon|\arity(\alpha)|\to\TSigma(X)$,
  \[ 
    A\models\alpha(e^\#\cdot\overline{\tau}\cdot f)
    \qquad\text{iff}\qquad
    A\models\alpha(e_{\tau}^\#\cdot f).
  \]
\end{cor}
\begin{proof}[Proof of \autoref{P:complete}]
  \textbf{{Soundness}} ($\Rightarrow$): Let $A$ be an algebra in $\V$
  and let $e\colon X\to A$ be a relation preserving assignment. We are
  going to show by simultaneous induction on derivations that
  \begin{enumerate}
  \item if $X\vdash\alpha(f)$ is derivable, then $e^\#\cdot f(i)$
    is defined for all $i\in\arity(\alpha)$, and
    \(
    A\models\alpha(e^\#\cdot f);
    \)
  \item if $X\vdash\isdef t$ is derivable, then $e^\#(t)$ is
    defined.
  \end{enumerate}
  For the base case of our inductive proof, first observe that if
  $X\vdash\alpha(f)$ is derivable via $(\mathsf{Ctx})$, then each
  $f(i)$ is a variable, so $e^\#\cdot f(i)=e\cdot f(i)$ is defined,
  and moreover $X\models\alpha(f)$, so $A\models\alpha(e\cdot f)$
  since $e\colon X\to A$ is relation preserving. Second, if
  $X\vdash\isdef t$ is derivable via $(\mathsf{Var})$, then $t$ is a
  variable in $X$ so that $e^\#(t)=e(t)$ is defined. We proceed by
  case distinction on the last rule applied in the
  derivation.%
  \smnote{I like the following much better as an enumerate; if we want to
    change the formatting we can easily do so by changing its settings.}

  \begin{enumerate}
\item[$(\mathsf{RelAx})$:] this follows from the fact that the given
  algebra $A$ is carried by an $\Horn$-model, in combination with
  \autoref{L:subs} and \autoref{C:subs}.%
  \smnote{Sorry, but I do not agree that it is 'immediate'; and by
    this terse style we missed that \autoref{L:subs} must also be used.}
  
  %

\item[$(\EAr)$:] Suppose that $X\vdash\isdef t$ is derived via the
  rule $(\EAr)$ so that $t = \sigma(h)$ for some $\sigma\in\Sigma$ and
  some map $h\colon\arity(\sigma)\to\TSigma(X)$. Then, by induction,
  we are given that $e^\#(h(i))$ is defined for all
  $i\in\arity(\sigma)$, and $A\models\alpha(e^\#\cdot(h\cdot g))$ for
  all $\alpha(g)\in\arity(\sigma)$.  It follows immediately from the
  definition of $e^\#$ that $e^\#(\sigma(h))$ is defined, as
  required.%
  \smnote{Why not $f$ in lieu of $h$ so that it matches precisely
    what's writen in the rule?}

\item[$(\mathsf{Mor})$:] Suppose that $X\vdash\alpha(f)$ is derived
  using the rule $(\mathsf{Mor})$. Then, for some $\sigma\in\Sigma$,
  we have that $f(i)= \sigma(g_i)$ for all~$i$, where
  $g_i\colon\arity(\sigma)\to\TSigma(X)$, and the premisses
  \[
    X\vdash\alpha(g_i(j))~(j\in\arity(\sigma))
    \qquad\text{and}\qquad
    X\vdash\isdef\sigma(g_i)~(i\in\arity(\alpha))
  \]
  of $(\mathsf{Mor})$ are derivable.  Then, by induction,
  $e^\#(\sigma(g_i))$ is defined for all $i\in\arity(\alpha)$.  By
  definition of $e^\#$, this means that
  $A\models\beta(e^\#\cdot (g_i\cdot k))$ for all
  $\beta(k)\in\arity(\sigma)$; that is,
  $e^\#\cdot g_i\in[\arity(\sigma), A]$ for all $i\in\arity(\alpha)$.
  Moreover, by the induction hypothesis applied to the first item, we
  have that $A\models\alpha(e^\#\cdot g_i(j))$ for all
  $j\in\arity(\sigma)$. In other words,
  $A\models\alpha(\pi_j\cdot (e^\#\cdot g_i))$ for all
  $j\in\arity(\sigma)$ (recall that
  $\pi_j(e^\#\cdot g_i) = e^\#\cdot g_i(j)$); that is,
  $[\arity(\sigma), A]\models\alpha(e^\#\cdot g_i)$. Now, using that
  $\sigma_A$ is a relation-preserving map $[\arity(\sigma), A]\to A$,
  it follows that $A\models\alpha(\sigma_A(e^\#\cdot g_i))$. But
  $\sigma_A(e^\#\cdot g_i)=e^\#(\sigma(g_i)) =e^\#\cdot f(i)$ for all
  $i\in\arity(\alpha)$, so we have that $A\models\alpha(e^\#\cdot f)$,
  as desired.

\item[$(\mathsf{Ax})$:] Now suppose that $X\vdash\alpha(f)$ is
  derivable via $(\mathsf{Ax})$ so that for some axiom
  $\Delta\vdash\alpha(h)$, where
  $h\colon\arity(\alpha)\to\TSigma(\Delta)$, and some substitution
  $\tau\colon\Delta\to\TSigma(X)$ we have $f=\tau\cdot h$.  The
  premisses of this rule have the following shape:
  \begin{enumerate}[label=(\arabic*)]
  \item
    $X\vdash\isdef\tau(y)$ for all $y\in\Delta$;
  \item
    $X\vdash\beta(\tau\cdot g)$ for all $\beta(g)\in\Delta$.
  \end{enumerate}
  By the inductive hypothesis, derivability of these premisses implies
  that $e^\#\cdot\tau(y)$ is defined for all $y\in\Delta$, and the
  assignment $e_{\tau}=e^\#\cdot\tau\colon\Delta\to A$ is
  relation-preserving. Since $A$, being in~$\V$, satisfies
  $\Delta\vdash\alpha(h)$, it follows that
  $A\models\alpha(e_\tau^\#\cdot h)$. By \autoref{C:subs}, we obtain
  $A\models\alpha(e^\#\cdot\tau\cdot h)$,
  i.e.~$A\models\alpha(e^\#\cdot f)$ as required.

\item[$(\IAr)$:] Finally, assume that $X\vdash\alpha(f)$ is derived
  using the rule $(\IAr)$. Then, for some substitution
  $\tau\colon\Delta\to\TSigma(X)$ and some $\V$-axiom
  $\Delta\vdash\beta(g)$, there exists $\sigma(h)\in\sub(g)$ and
  $\alpha(k)\in\arity(\sigma)$ such that%
  \smnote{What about $\arity(\sigma) \models \alpha(k)$ which is part
    of the side condition of the rule?}
  \[
    f=\arity(\alpha)\xra{k}\arity(\sigma)\xra{h}\TSigma(\Delta)\xra{\tau}\TSigma(X),
  \]
  and the premisses
  \begin{enumerate}[label=(\arabic*)]
  \item $X\vdash\gamma(\tau\cdot c)$ for all $\gamma(c)\in\Delta$;
  \item $X\vdash\isdef\tau(y)$ for all $y\in\Delta$
  \end{enumerate}
  of $(\IAr)$ are derivable.  Just as for the rule $(\mathsf{Ax})$,
  the induction hypothesis implies that $e^\#\cdot\tau(x)$ is defined
  for all $x\in X$, and $e_{\tau}=e^\#\cdot\tau\colon X\to A$ is a
  relation-preserving assignment.  Since $A$ lies in $\V$, we know
  that $A$ satisfies $\Delta\vdash\beta(g)$ so that
  $e_{\tau}^\#(g(i))$ is defined for all $i\in\arity(\beta)$ and
  $A\models\beta(e_{\tau}^\#\cdot g)$. Since $\sigma(h)\in\sub(g(i))$
  for some $i\in\arity(\beta)$, it follows from a routine induction on
  terms that $e_{\tau}^\#(\sigma(h))$ is defined (i.e.~we use that
  definedness is inherited by
  subterms).  
  Unwinding the definition of $e_{\tau}^\#$, we obtain that
  $e_{\tau}^\#(h(i))$ is defined for all $i\in\arity(\sigma)$, and
  $A\models\gamma(e_{\tau}^\#\cdot h\cdot c)$ for all
  $\gamma(c)\in\arity(\sigma)$. In particular,
  $A\models\alpha(e^\#_{\tau}\cdot h\cdot k)$. By~\autoref{C:subs},
  this implies that $A\models\alpha(e^\#\cdot \tau\cdot h\cdot k)$;
  since $f= \tau\cdot h\cdot k$, this means
  that~$A\models\alpha(e^\#\cdot f)$, as desired.
\end{enumerate}

\textbf{Completeness} ($\Leftarrow$):
Suppose that every algebra in $\V$ satisfies $X\vdash\alpha(f)$. 
Then, in particular, $\F X$ satisfies $X\vdash\alpha(f)$. Hence,
for the relation preserving assignment 
$\eta_X\colon X\to\TSigma(X)$ we have that $\eta^\#_X(f(i))$
is defined for all $i\in\arity(\alpha)$, and 
$\F X\models\alpha(\eta^\#\cdot f)$. By~\autoref{L:complete},
we have that $\eta^\#_X(f(i)) = [f(i)]$ for all $i\in\arity(\alpha)$.
In particular, $f$ factorizes as a map 
$f\colon\arity(\alpha)\to\Term_{\V}(X)\hookrightarrow\TSigma(X)$, 
and $\Term_{\V}(X)\models\alpha(f)$. That is,
$X\vdash\alpha(f)$ is derivable.
\end{proof}
\section{Details for \autoref{S:monad-theory}}
\begin{proof}[Proof of \autoref{L:TX-in-VT}]
  Let $X\in\BC$; we first show, for a given operation symbol $\sigma$ of arity $\Gamma$, 
  that $\sigma_{TX}\colon [\Gamma, TX] \to TX$
  is relation preserving. Indeed, we have
  \[
    \sigma_{TX} = \big(
    [\Gamma, TX]
    \xra{(-)^*}
    [T\Gamma, TX]
    \xra{\pi_x}
    TX
    \big),
  \]
  where $(-)^*$ is the Kleisli extension of the monad $T$ and
  $\pi_x(g) = g(x)$ (cf.~\autoref{D:internalhom}) is the evaluation
  map for $x = \sigma = \sigma(u_{\Gamma})$
  (\autoref{N:operation}). The Kleisli extension $(-)^*$ is relation
  preserving since $T$ is an enriched monad, and $\pi_x$ clearly is
  relation preserving by the definition of the relational structure on
  $[T\Gamma, TX]$. Thus $\sigma_{TX}$ is relation preserving.
  
  Now we are going to verify that $TX$ lies in
  $\V_{\bbT}$. First, observe that for all $\sigma\in\Sigma_{\bbT}$
  and all relation preserving assignments
  $f\colon\arity(\sigma)\to TX$, we have
  \begin{equation}\label{P:canonical-eq}
    f^\#(\sigma) = f^*(\sigma).			
  \end{equation}
  (On the left-hand side of the equation above, we view $\sigma$ as a
  term according to~\autoref{N:operation}.)  Indeed, where
  $|\arity(\sigma)|=\{x_i \mid i\in\arity(\sigma)\}$, we see that
  $f^\#(\sigma)$ is defined using that $f$ is relation preserving: for
  every $\Rels$-edge $\alpha(g)$ in $\arity(\sigma)$ we have
  $\alpha_{TX}(f^\# \cdot (u_{\arity(\sigma)}) \cdot g) =
  \alpha_{TX}(f \cdot g)$. Furthermore, using only the definitions of
  $f^\#$ and of $\sigma_{TX}$, we have
  \[
    f^\#(\sigma)
    =
    f^\#(\sigma(x_i))
    =
    \sigma_{TX}(f^\#(x_i))
    = \sigma_{TX}(f(x_i))
    = \sigma_{TX}(f)
    = f^*(\sigma).
  \]
  
  We will now see that $TX$ satisfies all axioms of $\V_{\bbT}$, as
  required. To this end, we make a case distinction on the axiom types
  described in \autoref{D:inducedtheory}:

  \begin{enumerate}
  \item 
    Given a $\Sigma$-relation $\Gamma\vdash\alpha(\sigma_i)$ such that
    $T\Gamma\models\alpha(\sigma_i)$ and a relation preserving map
    $f\colon\Gamma\to TX$, we first note that
    $TX\models\alpha(f^\#(\sigma_i))$ since $f^*\colon T\Gamma\to TX$
    is relation preserving and $f^*(\sigma_i)=f^\#(\sigma_i)$ for all
    $i\in\arity(\alpha)$ by~\eqref{P:canonical-eq}. Thus, to see that
    $TX$ satisfies $\Gamma\vdash\alpha(\sigma_i)$, it suffices to show
    that $f^\#(\sigma_i)$ is defined for all $i\in\arity(\sigma)$.
    This follows immediately from~\eqref{P:canonical-eq}.

  \item Now, given $\Delta\in\pres_{\lambda}$, a morphism
    $f\colon\Delta\to T\Gamma$, and an operation $\sigma\in T\Delta$,
    we verify that $TX$ satisfies
    $\Gamma\vdash f^*(\sigma)=\sigma(f)$.  To this end, let
    $m\colon\Gamma\to TX$ be a relation preserving assignment. Then
    $m^\#(f^*(\sigma))$ is defined; this is shown completely
    analogously as before.
    Furthermore, we have
    \begin{align*}
      m^\#(f^*(\sigma))
      &= m^*\cdot f^*(\sigma)  & \text{by~\eqref{P:canonical-eq}} \\
      &= (m^*\cdot f)^*(\sigma)& \text{by~\eqref{eq:kleisli-laws}} \\
      &= \sigma_{TX}(m^*\cdot f)&\text{def.~of $\sigma_{TX}$} \\
      &= \sigma_{TX}(m^\#\cdot f)&\text{by~\eqref{P:canonical-eq}} \\
      &= m^\#(\sigma(f))&\text{def.~of $m^\#$.} 
    \end{align*}

  \item Given $f\colon \Gamma \to TX$ relation preserving, we
    use~\eqref{P:canonical-eq} to obtain
    \[
      f^\#(\eta_\Gamma(x)) = f^*(\eta_\Gamma(x)) = f(x) =
      f^\#(x).
      \tag*{\qedhere}
    \]
  \end{enumerate}
\end{proof}

\begin{proof}[Proof of \autoref{T:monad-theory}]
  We first show that every canonical algebra
  $TX$ has the universal property of a free algebra in
  $\V_{\bbT}$ with respect to the monad unit $\eta_X\colon X\to
  TX$.  We split this into two parts:
  \begin{enumerate}
  \item Assume that
    $X=\Gamma\in\pres_{\lambda}$ and let $f\colon\Gamma\to
    A$ be a relation preserving map; we proceed to show that there
    exists a unique homomorphism $\bar{f}\colon T\Gamma\to
    A$ such that
    $f=\bar{f}\cdot\eta_{\Gamma}.$ To this end, let
    $\bar{f}$ be the map defined by
    \[
      \bar{f}(\sigma):=\sigma_A(f)
    \]
    for all $\sigma\in T\Gamma$. We first prove that $\bar f \cdot
    \eta_\Gamma = f$; for every $x \in \Gamma$ we have
    \begin{align*}
      \bar{f}\cdot\eta_{\Gamma}(x)
      &=
      \bar{f}(\eta_{\Gamma}(x))
      \\
      &=
      (\eta_{\Gamma}(x))_A(f)
      &
      \text{def.~of $\bar f$} \\
      &=
      (\eta_{\Gamma}(x))_A(f^\# \cdot u_\Gamma)
      &
      \text{def.~of $f^\#$}
      \\
      &=
      f^\#\big((\eta_\Gamma(x))(u_\Gamma)\big)
      &
      \text{def.~of $f^\#$} \\
      &=
      f^\#(x)
      &
      \text{since $A$ satisfies $\Gamma \vdash \eta_\Gamma(x) = x$}
      \\
      &=
      f(x)
      &\text{def.~of $f^\#$}.
    \end{align*}
    
    Next we prove that
    $\bar{f}$ is relation preserving. Indeed, given an edge
    $\alpha(\sigma_i)$ in $T\Gamma$, we have that
    $A$ satisfies
    $\Gamma\vdash\alpha(\sigma_i)$ since it is an algebra in
    $\V_{\bbT}$. In particular,
    $A\models\alpha_A(f^\#(\sigma_i))$. Moreover, we have
    \[
      f^\#(\sigma_i)
      =
      f^\#(\sigma_i(x_i))
      =
      (\sigma_i)_A(f(x_i))
      =
      \bar{f}(\sigma_i)
    \]
    for all $i\in\arity(\alpha)$. Hence  
    $A\models\alpha(\bar{f}(\sigma_i)).$ 

    Now we show that $\bar{f}$ is a homomorphism.  To this end, let
    $\tau\in\Sigma_{\bbT}$. We are going to show that the following
    square commutes
    \[
      \begin{tikzcd}[column sep = 40]
        {[\arity(\tau), T\Gamma]}
        \ar{r}{\tau_{T\Gamma}}
        \ar{d}[swap]{\bar f \cdot (-)}
        &
        T\Gamma
        \ar{d}{\bar f}
        \\
        {[\arity(\tau), A]}
        \ar{r}{\tau_A}
        &
        A        
      \end{tikzcd}
    \]
    Indeed, for every relation preserving $m\colon\arity(\tau)\to T\Gamma$ we have
    \begin{align*}
      \bar{f}(\tau_{T\Gamma}(m))
      &= \bar{f}(m^*(\tau))	& \text{def.~of $\tau_{T\Gamma}$} \\
      &= (m^*(\tau))_A(f)	&\text{def. of $\bar{f}$} \\
      &=  f^\#(m^*(\tau))	& \text{def. of $f^\#$} \\
      &=  f^\#(\tau(m))	& \text{A satisfies $\Gamma\vdash m^*(\tau)=\tau(m)$} \\
      &=  \tau_A(f^\#\cdot m)	& \text{def. of $f^\#$} \\
      &= \tau_A(\bar{f}\cdot m).
    \end{align*}
    For the last step we shall prove that
    $f^\# (m(x)) = \bar f (m(x))$ for every $x \in
    \arity(\tau)$. Indeed, using the definition of $f^\#$ we see that
    the operation symbol $\sigma = m(x)$, considered as the
    term~$\sigma(u_\Gamma)$ as in \autoref{N:operation} satisfies
    \[
      f^\#(\sigma(u_\Gamma))
      =
      \sigma_A(f^\#\cdot u_\Gamma)
      =
      \sigma_A(f)
      =
      \bar f(\sigma).
    \]

    As for the uniqueness, suppose that $\bar f\colon T\Gamma \to A$
    is a homomorphism such that $\bar f \cdot \eta_\Gamma = f$. Then
    the above square commutes for $\arity(\tau) = \Gamma$ which
    applied to $\eta_\Gamma \in [\Gamma, T\Gamma]$ yields for every
    $\sigma \in |T\Gamma|$:
    \begin{align*}
      \bar f(\sigma)
      &= \bar f (\eta^*_\Gamma(\sigma))
      & \text{by~\eqref{eq:kleisli-laws}}
      \\
      &= \bar f(\eta^\#_\Gamma(\sigma))
      & \text{by~\eqref{P:canonical-eq}}
      \\
      &=
      \bar f(\sigma_{T\Gamma}(\eta_\Gamma))
      &\text{def.~of $\eta_\Gamma$}
      \\
      &=\sigma_A(\bar f \cdot \eta_\Gamma)
      &\text{$\bar f$ homomorphism}
      \\
      &= \sigma_A(f)
      & \text{since $\bar f \cdot \eta_\Gamma = f$}.
    \end{align*}
    
  \item Now suppose that $X$ is an arbitrary object in $\BC$.  Using
    that $\BC$ is locally $\lambda$-presentable, we may express $X$ as
    a $\lambda$-directed colimit of $\lambda$-presentable objects, say
    $X=\colim\Gamma_i$.  Then, since $T$ is $\lambda$-accessible, we
    see that $TX=\colim T\Gamma_i$, and this lifts to a colimit in
    $\V_{\bbT}$ because the forgetful functor $\V_{\bbT}\to\BC_0$
    creates $\lambda$-directed colimits by \autoref{P:creation} and
    \autoref{rem:alg-sigma-props}\ref{item:alg-sigma-accessble}.

  \item To complete the proof we apply \autoref{R:kleisli-mnd}. The
    given monad $\bbT$ and the free algebra monad of $\V_\bbT$ share
    the same object assignment $X \mapsto TX$, and the family of
    morphisms $\eta_X$, as shown in the previous two items. It remains
    to prove that for every morphism $h\colon X \to TY$ the morphism
    $h^*\colon TX \to TY$ is a $\Sigma$-homomorphism. Then $\bbT$ and
    the free algebra monad of $\V_\bbT$ also share the same operator
    $h \mapsto h^*$. Given $\sigma \in |T\Gamma|$ we shall prove that
    the following square commutes:
    \[
      \begin{tikzcd}[column sep = 40]
        {[\Gamma, TX]}
        \ar{r}{\sigma_{TX}}
        \ar{d}[swap]{h^* \cdot (-)}
        &
        TX
        \ar{d}{h^*}
        \\
        {[\Gamma, TY]}
        \ar{r}{\sigma_{TY}}
        &
        TY
      \end{tikzcd}
    \]
    Indeed, given $f\colon \Gamma  \to TY$ we have
    \begin{align*}
      h^* \cdot \sigma_{TX}(f) &= h^* \cdot f^*(\sigma)
      & \text{def.~of $\sigma_{TX}$}
      \\
      &= (h^* \cdot f)^*(\sigma)
      & \text{by~\eqref{eq:kleisli-laws}}
      \\
      &=\sigma_{TY}(h^* \cdot f)
      & \text{def.~of $\sigma_{TY}$}.
      \tag*{\qedhere}
    \end{align*}
  \end{enumerate}
\end{proof}

\end{document}